\documentclass{ws-ijgmmp}

\newtheorem{assumption}[theorem]{Assumption}
 \RequirePackage[numbers]{natbib}
\usepackage{bm}
\usepackage{cancel}
\usepackage{mathtools}
\usepackage{xcolor}
\usepackage{macros}
\usepackage{subfiles}
\usepackage{cleveref} 
\usepackage{accents} 

\begin{document}

\markboth{G.Chirco, L. Malag\`o, G. Pistone}{Statisticle Bundle Dynamics for Probability Densities}

%
\catchline{}{}{}{}{}
%

\title{Lagrangian and Hamiltonian Dynamics for Probabilities on the Statistical Bundle}

\author{Goffredo Chirco}

\address{Dipartimento di Fisica Ettore Pancini, Universit\'a di Napoli Federico II\\
\& INFN Sezione di Napoli, Napoli, Italy\\
\email{goffredo.chirco@unina.it} }

\author{Luigi Malag\`o}

\address{Transylvanian Institute of Neuroscience\\
       Cluj-Napoca, Romania\\
malago@tins.ro}

\author{Giovanni Pistone}

\address{de Castro Statistics, Collegio Carlo Alberto,\\
    Torino, Italy.\\
giovanni.pistone@carloalberto.org }

\maketitle

\begin{history}
\received{(Day Month Year)}
\revised{(Day Month Year)}
\end{history}

\begin{abstract}
We provide an Information-Geometric formulation of accelerated natural gradient on the Riemannian manifold of probability distributions, which is an affine manifold endowed with a dually-flat connection. In a non-parametric formalism, we consider the full set of positive probability functions on a finite sample space, and we provide a specific expression for the tangent and cotangent spaces over the statistical manifold, in terms of a Hilbert bundle structure that we call the Statistical Bundle.
In this setting, we compute velocities and accelerations of a one-dimensional statistical model using the canonical dual pair of parallel transports and define a coherent formalism for Lagrangian and Hamiltonian mechanics on the bundle. We show how our formalism provides a consistent framework for accelerated natural gradient dynamics on the probability simplex, paving the way for direct applications in optimization.
\end{abstract}








\section{Introduction}

Gradient-based optimization and its application to large-scale statistical inference problems are a very active topic in machine learning~\cite{bottou|curtis|nocedal:18}.  In recent years, following the milestone work of~\cite{nesterov:83}, accelerated gradient methods~\cite{Nesterov:05,Nesterov:07,Nesterov:14} had significant impact in optimization, given their ability to improve their convergence rate compared to gradient-based algorithms, and in certain cases to yield optimal rates~\cite{nemirovsky|yudin:93}.

Despite the large amount of work on acceleration in optimization, and the intuitive explanation that the Nesterov method can be described in terms of momentum, there is the lack of a unifying theoretical framework for the derivations of accelerated methods from a unique underlying principle. Recently, \cite{WibisonoE7351} addressed this issue by proposing a generative approach for accelerated methods, where accelerated optimization is recast in the form of a variational problem. In this setting, the Euler-Lagrange equations derived from a family of \emph{Bregman Lagrangians} provide a continuous time analogue of the typical oracle's gradient descent dynamics which allows to generate a large class of accelerated methods, including non-Euclidean extensions.

In statistics and machine learning, inference procedures are often characterized by the minimization of a loss function defined over parameterized statistical models, accomplished by gradient descent methods. Under common regularity conditions, statistical models admit a manifold structure~\cite{amari:85,amari|nagaoka:2000} and the optimization process often benefits by being set in a Riemannian context~\cite{absil|mahony|sepulchre:2008}. Indeed, the Riemannian natural gradient~\cite{amari:1998natural}, that takes into account the metric of the space for the identification of the direction of steepest descent, has proved to provide benefits in terms of speed of convergence in optimization, compared to the plain vanilla Euclidean gradient.  The generalization of accelerated methods to Riemannian manifolds has been investigated more recently by several authors, including~\cite{liu|shang|cheng|cheng|jiao:17,zhang2018riemannian,ahn|sra:2020}.

In this paper, motivated by the approach developed by~\cite{WibisonoE7351}, suitable for an arbitrary Hessian metric over a convex set in $\reals^N$, and with the general aim to inquire about the relation between the \emph{geometric mechanics} and Information Geometry~\cite{doi:10.1063/1.4984941,leok|zhang:2017,pistone:2018Lagrange, felice2018dynamical},\footnote{The interest of dynamical systems on probability functions has raised in several areas, for example, Compartmental Models, Replicator Equations, Prey-Predator Equations, Mass Action Equations, Differential Games, beside in Optimization Methods and Machine Learning Theory.} we define a theoretical framework for accelerated continuous-time dynamics over statistical manifolds, leading to the acceleration of the natural gradient. Notably, our framework defines Lagrangian and Hamiltonian mechanics of statistical manifolds, starting from generic diverge functions between probability distributions, hence including the Bregman Lagrangian as a special case.

Our construction is not limited to statistical manifolds, but it can be applied to any dually-flat geometry characterizing Hessian manifolds. Specifically, by considering functions defined over the natural parameters of the exponential family, which is a convex set by definition, together with the Fisher information matrix as a metric, we recover as a special case an instance of the framework developed in~\cite{WibisonoE7351}. 

Two specific qualifications characterise our approach.  Information Geometry, as firstly formalized by \cite{amari|nagaoka:2000}, views \emph{parametric} statistical models as a manifold endowed with a Riemannian metric and a family of dual connections, the $\alpha$-connections.  Differently, we consider the full set of positive probability functions on a finite sample space and discuss Information Geometry in the non-parametric geometric language (cf.  \cite{lang:1995,klingenberg:1995}). In data analysis, the non-parametric statistical study of compositional data has been started by \cite{aitchinson:1986}. We use here the simplest instance of non-parametric Information Geometry as it is described in the review papers by \cite{pistone:2013GSI,pistone:2020-NPCS}.%

The second and most qualifying choice, consists in considering Information Geometry as defined on a linear bundle, not just on a manifold of probability densities. In classical mechanics, the study of the evolution of a system requires both position $q$ and velocities $\dot{q}$, or position and conjugate momenta $p$ in the \emph{phase space} description. Lagrangian and Hamiltonian mechanics are defined, respectively, on the tangent and co-tangent bundle of a finite-dimensional Riemannian manifold \cite[Ch.~III-IV]{arnold:1989}.  Similarly, in statistics, we argue that the study of probability evolution is most naturally described by a fiber bundle structure comprising couples of probability densities and associated scores $\velocity q$ ($\log$-derivatives), or equivalently by densities and score conjugate momenta $\eta$. We call such a bundle the Statistical Bundle \cite{pistone:2018Lagrange}. This idea should be compared with the use of the Grassmannian manifold, as defined, for example in \cite{absil|mahony|sepulchre:2008}, to describe the various centering of the space of the sufficient statistics of an exponential family \cite{malago|pistone:2014Entropy, michalek|sturmfels|uhler|zwiernik:2016}.

In the case of strictly positive densities, the affine geometry of the Statistical Bundle is fully characterised by an affine atlas and a couple of dual parallel transports. With that, one easily computes the form of all the relevant second-order quantities which are necessary to define a notion of acceleration for in the space of probability densities.

In this setting, first order differential equations are confirmed to correspond to replicator equations, as discussed, for example, in \cite[\S6.2]{Ay|Jost|Le|Schwachhofer:2017IGbook}. However, we are able to show that also the second-order Euler-Lagrange equation (as well as the Hamiltonian equation) can be expressed as a \emph{system} of replicator equations. In this sense, the statistical bundle framework provides a possible solution to the problem of second-order evolution equation on the simplex, which has been raised in the optimization literature.

The paper is organised as follows.  In \cref{sec:statistical-bundle}, we introduce the non-parametric description of the statistical bundle and the maximal exponential family. We define a convenient \emph{full bundle} extension for this structure, which carries tuples of both exponential and mixture fibers at each point.  In \cref{sec:secondbundle}, we recall the main features of the Hessian geometry of the maximal exponential family. We focus on the second-order geometry, introducing consistent notions of velocity, covariant derivative, and acceleration on higher order statistical bundles.  In section \ref{sec:natural-gradient}, we generalize the computation of the natural gradient to the \emph{Lagrangian} and the \emph{Hamiltonian} function on the full bundle.  Therefore, in \cref{sec:mechanics}, we gathered all the necessary structure to define a mechanics of the probability simplex. We define an action integral in terms of a generic notion on Lagrangian function on the statistical bundle. We can then derive the Euler-Lagrange equation via a standard variational approach on the simplex \cite{pistone:2018Lagrange}.  We define a Legendre transform, hence we derive the Hamilton equations.  As a starting point for our analysis, we look at the dynamics induced by a standard, though \emph{local} here, particle Lagrangian obtained from the \emph{quadratic form} on the statistical bundle, where the role of the point particle is played by a probability density as a point on the statistical manifold.  We take the quadratic particle Lagrangian as a \emph{quadratic} approximation of a Kullback-Leibler (KL) divergence function. We focus on the formal construction of a Lagrangian function from a divergence and we setup the study of the dynamics induced by a family of parameterised KL divergence Lagrangians. Finally, motivated by our interest in applications in optimization, in \cref{sec:jordan} we consider the case of a time-dependent damped extensions of the KL Lagrangian and we apply the Lagrange-Hamilton duality to provide a first realization on the statistical bundle of the variational approach to accelerated optimisation methods recently proposed in \cite{WibisonoE7351}. We end with a brief discussion in \cref{sec:discussion}.  In the appendices we include examples, comments, and additional derivations. Therein, we provide the full analytic solution of the geodesic motion for the free quadratic Lagrangian, complete examples of both quadratic and KL Lagrangian and Hamiltonian flows on the bundle, as well as the explicit expressions of the ODEs systems associated to the dynamics derived in \cref{sec:mechanics} and \cref{sec:jordan}.

\section{Statistical Bundle}
\label{sec:statistical-bundle}

We work on a finite sample space $\Omega$, $\#\Omega=N$ with $x \in \Omega$. However, we are careful in using a general language that could be used in the study of an infinite or continous state space. Probabilities on a finite sample space can be presented as probability functions, that is real non-negative vectors $P = (P(x) \colon x \in \Omega)$ whose components sum to 1, $\sum_x P(x) = 1$. In this presentation, the uniform probability function is $\mu$, $\mu(x) = \frac1N$, $x \in \Omega$. The set of all probability functions is the \emph{probability simplex} $\Delta(\Omega)$, which is also described as the convex set in $\reals^\Omega$ generated by the Dirac functions $\delta_x$, $x \in \Omega$. In this paper, we focus on strictly positive probability functions, that is our base set is the relative interior $\Delta^\circ(\Omega)$ of the probability simplex. A random variable is a real function $f$ on $\Omega$ and its expected value with respect to the probability function $P$ is $\expectat P f =\sum _ x f(x)\, P(x)$.
Equivalently, a probability function $P$ on $\Omega$ is described by its density $p$ with respect to the uniform probability function, $P(x) = p(x) / N$. A random variable is centered with respect to $p$ if $\expectat p f=0$. 

Given a probability density $p$, we shall look systematically to a random variable $f$ as the sum of its $p$-mean value $\expectat p f$ and the fluctuation $(f - \expectat p f)$. We define the entropy function to be $\entropyof p = - \expectat p {\log p}$, so that the entropy of the uniform probability density is 0. The Kullback-Leibler divergence is $D(p,q) = \expectat p {\log\frac p q}$, so that $D(p, 1) = - \entropyof p$.
 We represent the open probability simplex $\Delta^\circ(\Omega)$ as the \emph{maximal exponential family} $\maxexpat \mu$ in the sense that each strictly positive probability density $p > 0$ can be written in exponential form as $p \propto \euler^f$, where $f$ is identified up to a constant. Uniqueness of the exponent can be obtained in at least two ways, both relevant for our construction. For each given reference density $p \in \maxexpat \mu$, considered as a reference state, one can write either
\begin{equation}\label{eq:pistone}
  q(x) = \expof{u(x) - K_p(u)} \cdot p(x) \ , 
\end{equation}
where $\expectat p u = 0$ is a constraint and the normalising constant becomes $K_p(u) = \log \expectat p {\euler^u} = D(p,q)$, or
\begin{equation}\label{eq:landau}
  q(x) = \expof{v(x) + H_p(v)} \cdot p(x) \ ,
\end{equation}
where $ \expectat q v = 0$ is a constraint and the normalizing constant becomes $H_p(v) = - \log \expectat q {\euler^v} = D(q,p)$.

In the first case \eqref{eq:pistone}, the set of all possible $u$'s is the vector space of $p$-centered random variables and the inverse mapping $q \mapsto u = \log q/p - \expectat p {\log q/p}$ provides a chart of $\maxexpat \mu$. In the second case \eqref{eq:landau}, the co-domain for the mapping $q \mapsto v = \log q / p - \expectat q {\log q/p}$ is not a vector space. Notice the equalities $u - v = D(p, q) + D(q,p)$ and $\expectat q u = K_p(u) + H_q(v)$.

The \emph{statistical bundle} with base $\maxexpat \mu$ is
\begin{equation}\label{eq:statisticalbundle}
    \expbundleat \mu = \setof{(q,v)}{q \in \maxexpat \mu, \expectat q v = 0} \ .
\end{equation}
The bundle projection is $\pi(q,v) = q$. The mapping $q \mapsto v = \log q/p - D(q, p)$, uniquely defined by \cref{eq:landau}, provides a section of the statistical bundle. Each fiber $\expfiberat p \mu$ is the vector space of fluctuations with respect to $p$. The statistical bundle is designed to allow for the discussion of the time evolution of curves $t \mapsto \gamma(t) = (q(t), v(t))$ in a space of states involving both probability distributions and fluctuations.

The statistical bundle is a semi-algebraic subset of $\reals^{2N}$, namely the open subset of the $(N-1)$-Grassmannian defined by
\begin{equation*}
(q,v) \in \expbundleat \mu \Leftrightarrow 
  \begin{cases}
    &\sum_{x\in\Omega} q(x) = N  \ , \\
    &\sum_{x\in\Omega} v(x)q(x) = 0 \ , \\
    &q(x) > 0 \quad x \in \Omega \ .
  \end{cases}
\end{equation*}
In the following, we will retain the manifold structure induced by $\reals^{2N}$ and proceed by adding further structure. 

It will be useful to distinguish between the \emph{exponential} (statistical) bundle $\expbundleat \mu$ and the \emph{mixture} (statistical) bundle denoted by $\mixbundleat \mu$ and introduce a duality pairing. The two bundles are algebraically equal, but they will carry different affine geometries. The duality pairing is defined on each fiber at $q$ by
\begin{equation}
 \mixfiberat q \mu \times \expfiberat q \mu \ni (\eta,v) \mapsto \scalarat q \eta v = \expectat q {\eta v} \ . 
\end{equation}
We will consider a mixture chart defined on $\maxexpat \mu$ by $q \mapsto \frac q p -1 \in \mixfiberat p \mu$. The exponential chart previously defined requires positive densities, while the mixture chart is in fact defined on all random variable whose sum of values equal 1, without any restriction on the sign. The relevance of the joint consideration of the two types of chart is shown by the following computation. If $q$ is represented by $u = \log \frac q p - \expectat p {\log \frac q p}$ and by $v = \frac q p -1$, it holds
\begin{equation*}
\scalarat p u v = \expectat p {\left(\log \frac q p - \expectat p {\log \frac q p}\right)\left(\frac q p -1\right)} = D(q,p) + D(p,q) \ .
\end{equation*}

The geometry of statistical models is frequently identified with a Riemannian geometry whose metric is given by the so-called Fisher metric. Such a Riemannian metric is derived either from the Hessian matrix of the log-likelihood of the model or, in a non parametric way, from the metric of the unit sphere via a square root embedding~\cite{amari|nagaoka:2000}. We discuss here the relation of this approach to our construction of statistical manifold.

Consider the positive quadrant $S_>(2)$ of the sphere of radius 2 of $L^2(\mu)$ and its tangent bundle  $TS_>(2)$. An element $(\rho,w)$ in the bundle is characterised by the equations $\rho(x) > 0$, $\sum_x \rho(x)^2/N = 2^2$, $\sum_x w(x)\rho(x)/N = 0$. In this finite case, the set of positive probability densities $\maxexpat \mu$ is a convex subset of $L^2(\mu)$. An element $(p,v)$ of the tangent bundle $T \maxexpat \mu$ is characterised by the equations $p(x) > 0$, $\sum_x p(x) / N = 1$, $\sum _x v(x) / N = 0$. The tangent mapping of the so-called square embedding $\rho \mapsto \rho^2/4$  identifies the two bundles: 
\begin{equation}\label{eq:spheretosimplex}
  TS_>(2) \ni (\rho,w) \mapsto \left(\frac 14 \rho^2, \frac12 \rho
    w\right) = (q,v) \in T\maxexpat \mu \ .
\end{equation}
The inverse transformation is $(p,v) \mapsto (2\sqrt p,v/\sqrt p)$ so that the inner product of $L^2(\mu)$ restricted to the fiber $T_\rho S_>(2)$ is pushed forward to the (trivial) fiber $T_p\maxexpat \mu$, $\rho = 2 \sqrt p$, as 
\begin{equation*}
\scalarat \mu {\frac {v_1}{\sqrt p}}{\frac {v_2}{\sqrt p}} = \sum _ x \frac {v_1(x) v_2(x)}{p(x)} \frac1N \ ,
\end{equation*}
that is, the Fisher metric. If we now move from the statistical bundle to the tangent bundle by the trivialization mapping
\begin{equation}\label{eq:simplextobundle}
\mixbundleat \mu = \expbundleat \mu \ni (p,u) \mapsto (p,pu) = (p,v) \in T\maxexpat \mu
\end{equation}
the push forward of the Fisher metric to $\expbundleat \mu$ is
\begin{equation*}
\sum _ x \frac {(p(x) v_1(x)) (p(x) v_2(x))}{p(x)} \frac1N = \scalarat p {v_1}{v_2} \ ,
\end{equation*}
that is, the restriction of the inner product of $L^2(p)$. In conclusion, the same metric structure has three different canonical expressions, according to the choice of the bundle among isomorphic expressions, 
\begin{equation*}
 TS_>(2) \leftrightarrow T\maxexpat \mu \leftrightarrow \expbundleat \mu = \mixbundleat \mu \ . 
\end{equation*}
Our choice of the statistical bundle as basic representation is motivated by two arguments. First, the representation on the sphere produces results that do not have a clear statistical interpretation. Second, the description of the affine connections is especially simple in the statistical bundle.

Let us discuss briefly the issue of the parameterization of the statistical bundle. For the purpose of doing computations, it could be convenient to represent probabilities as probability functions in the open probability simplex instead of probability densities with respect to the uniform probability. Moreover, one could restrict to the set of probabilities $\Gamma_{N-1}$ defined by $N-1$ parameters $\theta_j > 0$, $\sum_j \theta_j < 1$. In this presentation the tangent space is the space $\reals^{(N-1)}$. A straightforward computation shows that the Fisher metric is represented in the standard basis by the matrix $I(\bm\theta) = \left(\diagof {\bm\theta} - {\bm\theta} {\bm\theta}^T\right)^{-1}$.

We now proceed to introduce the affine geometry of the statistical bundle. Recall that we look at the inner product on the fibers as a duality pairing between $\mixfiberat q \mu$ and $\expfiberat q \mu$. This point of view allows for a natural definition of a dual covariant structure. Namely, we define two affine transports among the fibers of each of the statistical bundles. See the parametric version in \cite{amari|nagaoka:2000} and the first non-parametric version in~\cite{gibilisco|pistone:98}.

For each random variable $u\in \expfiberat p \mu$, it holds
\begin{equation}
\expectat q { u-\expectat q {u}}=0 \quad \text{and} \quad  \expectat q {\frac{p}{q}\,u} =0,
\end{equation}
so that both $u- \expectat q{u}$ and $\frac{p}{q}\,u$ belong to $\expfiberat q \mu$. This prompts for the following definition.

\begin{definition}
  The \emph{exponential transport} is defined for each $p,q \in \maxexpat \mu$ by 
\begin{equation}
\etransport p q \colon \expfiberat p \mu \to \expfiberat q \mu \ , \quad \etransport p q v = v - \expectat q v \ ,
\end{equation}
while the \emph{mixture transport} is
\begin{equation}
\mtransport p q \colon \mixfiberat p \mu \to \mixfiberat q \mu \ , \quad \mtransport p q \eta = \frac p q \eta  \ .
\end{equation}
\end{definition}

The e-transport and the m-trasport are semi-groups of affine transformations which are compatible with the statistical bundle and are dual of each other with respect to the scalar product on each fiber. In particular, we have $\etransport p q\, \etransport r p= \etransport r q $ and $\mtransport p q\, \mtransport r p= \mtransport r q $, respectively.

The following properties are easily proved by applying the semi-group property of the transports and the definition of pairing $\scalarat q  A B= \expectat q {A\,B}$.

\begin{proposition}
  The two transports defined above are conjugate with respect to the duality pairing,
\begin{equation} \label{eq:transports-duality}
\scalarat q {\mtransport p q \eta} v = \scalarat p \eta {\etransport q p v} \ , \quad \eta \in \mixfiberat p \mu, v \in \expfiberat q \mu \ .  
\end{equation}
Moreover, it holds
\begin{equation}\label{conj-scalarat}
\scalarat q {\mtransport p q \eta} {\etransport p q v} = \scalarat p
\eta v \ , \quad \eta \in \mixfiberat p \mu, v \in \expfiberat p \mu \ .    
\end{equation}
\end{proposition}

We now use these notions of transport to define a special affine atlas of charts, which will then be used to introduce the \emph{affine manifold} structure providing the set-up of Information Geometry in this setting \cite{pistone:2020-NPCS}.

\begin{definition}
The \emph{exponential atlas} of the exponential statistical bundle $\expbundleat \mu$ is the collection of charts given for each $p \in \maxexpat \mu$ by
 \begin{equation}\label{eq:expcharts1}
   s_p \colon \expbundleat \mu \ni (q,v) \mapsto (s_p(q),\etransport q p v) \in \expfiberat p \mu \times \expfiberat p \mu \ , 
 \end{equation}
where 
\begin{equation}\label{eq:expcharts2}
  s_p(q) = \log \frac qp - \expectat p {\log \frac qp} \ .  
\end{equation}
\end{definition}

As $s_p(p,v) = (0,v)$, we say that $s_p$ is the chart \emph{centered
  at $p$}. If $s_p(q) = u$, from \cref{eq:expcharts2} follows the exponential form of $q$ as a density with respect to $p$,
namely $q = \euler^{u - \expectat p {\log \frac qp}} \cdot p$. As
$\expectat \mu q = 1$, then
$1 = \expectat p {\euler^{u - \expectat p {\log \frac pq}}} =
\expectat p {\euler^u} \euler^{-\expectat p {\log \frac pq}}$, so that
the \emph{cumulant function} $K_p$ is defined on $\expfiberat p \mu$
by
\begin{equation}
  K_p(u) = \log \expectat p {\euler^u} = \expectat p {\log \frac pq} = D(p,q) \ ,
\end{equation}
that is, $K_p(u)$ is the expression in the chart at $p$ of Kullback-Leibler divergence of $q \mapsto D(p,q)$, and we can write
\begin{equation}
q = \euler^{u - K_p(u)} \cdot p = e_p(u) \ .  
\end{equation}

In conclusion, the \emph{patch centered at $p$} is

\begin{equation}
 s^{-1}_p = e_p \colon (\expfiberat p \mu)^2 \ni (u,v) \mapsto (e_p(u), \etransport p {e_p(u)} v) \in \expbundleat \mu \ .
\end{equation}

In statistical terms, the random variable $\logof{q/p}$ is the relative point-wise information about $q$ relative to the reference $p$, while $s_p(q)$ is the deviation from its mean value at $p$.

The expression of the other divergence in the chart centered at $p$ is 
\begin{equation}
D(q,p) = \expectat q {\log \frac qp} = \expectat q {u - K_p(u)} = \expectat q u - K_p(u) \ .
\end{equation}

\begin{definition}
The \emph{dual  atlas} of the mixture statistical bundle $\expbundleat \mu$ is the collection of charts given for each $p \in \maxexpat \mu$ by
 \begin{equation}\label{eq:mixcharts1}
   \eta_p \colon \mixbundleat \mu \ni (q,w) \mapsto \left(s_p(q),\mtransport q p w\right) \in \expfiberat p \mu \times \mixfiberat p \mu \ . 
 \end{equation}
\end{definition}

We say that $\eta_p$ is the chart \emph{centered at $p$}. The \emph{patch centered at $p$} is
\begin{equation}
 \eta^{-1}_p \colon \expfiberat p \mu \times \mixfiberat p \mu \ni (u,v)
 \mapsto \left(e_p(u), \mtransport p {e_p(u)} v\right)  \in \mixbundleat \mu \ .
\end{equation}

we will see that the affine structure is defined by the affine atlases.

\begin{remark}As an aside, we underline that there is a further structure of interest, namely the \emph{Hilbert bundle}. This is obtained by considering each fiber as an expression of the tangent bundle and taking the inner product $\scalarat p \cdot \cdot$ as a Riemannian metric. In this case, the relevant Levi-Civita connection associated to the metric induces a parallel transport and a geometry on the bundle which is not affine in our sense. The push forward of the Riemannian parallel transport can be computed explicitly so to have the isometric property $\scalarat q v w = \scalarat p {\transport q p v}{\transport q p w}$ \cite{pistone:2020-NPCS}. The notion of Hilbert bundle was introduced originally by \cite{kumon|amari:1988} and developed by \cite{amari:87dual} as a general set-up for the study of sub-models. See also the discussion in the monograph by \cite[\S~10.1-2]{kass|vos:1997}. We will not make any use of this in the following.
\end{remark}

In some cases, especially in discussing higher-order geometry, we will need bundles whose fibers are the product of multiple copies of the mixture and exponential fibers. As a first example, the \emph{full bundle} is
\begin{equation*}
  \fullbundleat \mu = \setof{(q,\eta,w)}{q \in \maxexpat \mu, \eta \in
    \mixfiberat q \mu, w \in \expfiberat q \mu} \ .
\end{equation*}
In general, $\prescript{h}{}\!S^k$ will denote $h$ mixture factors and $k$ exponential factors. Note $\prescript{1}{}\!S^0 = \prescript{*}{}\!S$ and $\prescript{0}{}\!S^1 = S$.

\section{Hessian Structure and Second-Order Geometry}
\label{sec:secondbundle}

In the construction of the statistical bundle given above, we were
inspired by the original Amari's Information Geometry
\cite{amari|nagaoka:2000} in that we have shown that the statistical
bundle is an extension of the tangent bundle of the Riemannian
manifold whose metric is the Fisher metric and, moreover, we have
provided a system of dually affine parallel transports. In this
section, we proceed by introducing a further structure, namely, we
show that the base manifold $\maxexpat \mu$ is actually a Hessian
manifold with respect to any of the convex functions
$K_p(u) = \log \expectat p {\euler^u}$, $u \in \expfiberat p \mu$, see \cite{shima:2007}. Many useful computations in
classical Statistical Physics and, later, in Mathematical Statistics,
have been actually performed using the derivatives of a master convex
function, that is, using the Hessian structure.

The connection is established by the fact that the $n$-th differential of $K_p(u)$ at $u$ in the direction $h \in \expfiberat p \mu$, given by the $n$-linear continuous form $D^n K_p(u)$ applied to $\underbrace{(h, \dots, h)}_{n \text{ times}}$, is the $n$-th cumulant of $\etransport p {e_p(u)} h$ under the probability density $q$, see \cite[Proposition 2.4]{pistone|sempi:95}. In particular, the following equations can be easily derived
\begin{gather}
  {D K_p(u)[h]} = \expectat {e_p(u)} h \ , \label{eq:K1}\\
 D^2K_p(u)[h_1,h_2] = \scalarat {e_p(u)}{\etransport p {e_p(u)} h_1}{\etransport p {e_p(u)} h_2} \ , \label{eq:K3}\\    
    D^3K_p(u)[h_1,h_2,h_3] = 
    \expectat {e_p(u)} {(\etransport p {e_p(u)} h_1)(\etransport p {e_p(u)} h_2)(\etransport p {e_p(u)} h_3)} \ , \label{eq:K4}
  \end{gather}
For Hessian manifolds the second-order geometry gets fully encoded in the first three cumulants. In the first cumulant the directional derivative is computed by an expected value; the second cumulant defines a metric bilinear form and allows to compute inner products as covariances; the third cumulant directly relates to the computation of the covariant derivative for Hessian manifolds.
With such computational tools, we can proceed to discuss the kinematics of the
statistical bundles.

\subsection{Velocities and Covariant Derivatives}

Let us compute the expression of the velocity at time $t$ of a smooth curve
\begin{equation}
  \label{eq:3}
t \mapsto \gamma(t) = (q(t), w(t)) \in \expbundleat \mu  
\end{equation}
in the exponential chart centered at $p$. The expression of the curve is
\begin{equation}
\gamma_{p}(t) = \left(s_p(q(t)),\etransport {q(t)}p w(t)\right) \ ,  
\end{equation}
and hence we have, by denoting the ordinary derivative of a curve in $\reals^N$ by the dot,
\begin{multline}\label{eq:deriv1}
  \derivby t  s_p(q(t)) = \derivby t \left(\log \frac {q(t)} p - \expectat p {\log \frac {q(t)} p}\right) = \frac {\dot q(t)}{q(t)} - \expectat p {\frac {\dot q(t)}{q(t)}} = \\ \etransport {q(t)}{p} \frac {\dot q(t)}{q(t)} = \etransport {q(t)} p \derivby t \log q(t) \ ,
\end{multline}
and 
\begin{equation}\label{eq:deriv2}
  \derivby t  \etransport {q(t)} p w(t) = \derivby t \left(w(t) - \expectat p {w(t)}\right) = \dot w(t) - \expectat p {\dot w(t)} \ . 
\end{equation}

There is a clear advantage in expressing the tangent at each time $t$ in the moving frame centered at the position $q(t)$ of the curve itself. Because of that, we define the \emph{velocity} of the curve
\begin{equation}\label{eq:expcurve}
t \mapsto q(t) = \euler^{u(t) - K_p(u(t))}\cdot p \ , \quad u(t) = s_p(q(t)) \ , 
\end{equation}
to be
\begin{multline}
\label{eq:exponential-derivative}
\velocity q(t) =  \etransport p {q(t)} \derivby t s_p(q(t)) = \dot u(t) - \expectat {q(t)}{\dot u(t)} = \dot u(t) - D K_p(u(t))[\dot u(t)] = \\ 
\derivby t \log q(t) = \frac {\dot q(t)}{q(t)} \ . 
    \end{multline}

It follows that $t \mapsto (q(t),\velocity q(t))$ is a curve in the statistical bundle whose expression in the chart centered at $p$ (the reference density in \cref{eq:expcurve}) is $t \mapsto (u(t),\dot u(t))$. In fact,
\begin{equation}
\etransport {q(t)} p \left(\dot u(t) - D K_p(u(t))[\dot u(t)]\right) = \dot u(t) \ .
\end{equation}

The mapping $q \mapsto (q,\velocity q)$ is a lift of the curve to
the statistical bundle.

The velocity as defined above is nothing else but the \emph{score
  function} of a one-dimensional parametric statistical model, see,
for example, the contemporary textbook by \cite{efron|hastie:2016}, \S
4.2.


Let us turn to the interpretation of the second component in
\cref{eq:deriv2}. Given the exponential parallel transport, we define
a \emph{covariant derivative} by setting
\begin{equation}\label{eq:mixture-derivative}
\Derivby t w(t) = \etransport p {q(t)} \derivby t \etransport {q(t)} p
w(t) = \etransport p {q(t)} \Big( \dot w(t) - \expectat p {\dot w(t)}\Big) = \dot
w(t) - \expectat {q(t)} {\dot w(t)} \ .  
\end{equation}
Throughout the paper, the notation $\,\Derivby t \,$ denotes  the
covariant time derivative in a given transport or connection, whose choice will depend on the context.

Let us do the computation in the \emph{dual bundle}. The curve now is
$\zeta(t) = (q(t),\eta(t))$ and the expression of the second component
is $\mtransport {q(t)} p \eta(t) = \frac {q(t)} p \eta(t)$. This gives
\begin{equation}
\derivby t \mtransport {q(t)} p \eta(t) = \derivby t \frac {q(t)} p
\eta(t) = \frac 1p\left(\dot q(t) \eta(t) + q(t)\dot \eta(t)\right) \ ,
\end{equation}
which, in turn, gives the dual covariant derivative
\begin{equation}
\Derivby t \eta(t) = \mtransport p {q(t)} \derivby t \mtransport {q(t)} p
\eta(t) = \frac {p}{q(t)} \frac 1p \left(\dot q(t) \eta(t) + q(t)\dot \eta(t)\right)
= \velocity q(t) \eta(t) + \dot \eta(t) \ .\end{equation}

The couple of covariant derivatives of
\cref{eq:exponential-derivative,eq:mixture-derivative} are compatible
with the duality pairing, as the following proposition shows.

\begin{proposition}[Duality of the covariant derivatives]\label{prop:covariantduality}
For each smooth curve in the full statistical bundle,
\begin{equation*}
  t \mapsto (q(t),\eta(t),w(t)) \in \fullbundleat \mu \ ,
\end{equation*}
it holds
\begin{equation}  \label{eq:d-inner}
  \derivby t \scalarat {q(t)} {\eta(t)}{w(t)} = 
\scalarat {q(t)} {\Derivby t \eta(t)}{w(t)} + \scalarat {q(t)}
{\eta(t)} {\Derivby t w(t)} \ .
\end{equation}
\end{proposition}

\begin{proof}
  The proof is a simple computation based on \cref{{eq:transports-duality}}.
  \begin{multline*}
  \derivby t \scalarat {q(t)} {\eta(t)}{w(t)} = \derivby t \scalarat {p} {\mtransport {q(t)} p \eta(t)}{\etransport {q(t)} p w(t)} = \\
  \scalarat {p} {\derivby t \mtransport {q(t)} p \eta(t)}{\etransport {q(t)} p w(t)} + \scalarat {p} {\mtransport {q(t)} p \eta(t)}{\derivby t \etransport {q(t)} p w(t)} = \\
  \scalarat {q(t)} {\mtransport p {q(t)} \derivby t \mtransport {q(t)} p \eta(t)}{w(t)} + \scalarat {q(t)} {\eta(t)}{\etransport p {q(t)}\derivby t \etransport {q(t)} p w(t)} = \\
\scalarat {q(t)} {\Derivby t  \eta(t)}{w(t)} + \scalarat {q(t)}
{\eta(t)} {\Derivby t  w(t)} \ .
\end{multline*}
\end{proof}

Let us now look at the duality pairing
$(\square,\lozenge) \mapsto \scalarat q \square \lozenge$ as an inner
product $(\bigcirc,\bigcirc) \mapsto \scalarat q \bigcirc \bigcirc$ on
the Hilbert space $L^2_0(q)$. As topological vector spaces, we can use
the identification $L^2_0(q) = \mixfiberat q \mu = \expfiberat q \mu$,
so that we can consider the full bundle as an Hilbert bundle. Let be
given a smooth curve in such a bundle,
$t \mapsto (q(t),\alpha(t),\beta(t))$. Because now the two statistical
bundles are identified, we are bound to provisionally use different
notations for the two covariant derivatives.

By using the symmetry, we get
\begin{multline*}
  \derivby t \scalarat {q(t)} {\alpha(t)}{\beta(t)} = 
\scalarat {q(t)} {\mDerivby t \alpha(t)}{\beta(t)} + \scalarat {q(t)}
{\alpha(t)} {\eDerivby t \beta(t)} = \\
\qquad \quad \scalarat {q(t)} {\eDerivby t \alpha(t)}{\beta(t)} + \scalarat {q(t)}
{\alpha(t)} {\mDerivby t \beta(t)} 
= \\
\scalarat {q(t)} {\oDerivby t \alpha(t)}{\beta(t)} + \scalarat {q(t)}
{\alpha(t)} {\oDerivby t \beta(t)} \ ,
\end{multline*}
where
\begin{equation*}
\oDerivby t = \frac12 \left(\mDerivby t + \eDerivby t\right) \ .
\end{equation*}

Up now, we have defined the following derivation operators on the
statistical bundles:

\begin{enumerate}
\item A velocity $\velocity q(t) = \derivby t \log q(t)$, which is the expression in the moving frame of the derivative.
\item An exponential covariant derivative $\Derivby t w(t) = \eDerivby
  t w(t) = \etransport p {q(t)} \derivby t \etransport {q(t)} p w(t)$.
\item A mixture covariant derivative, $\Derivby t \eta(t) = \mDerivby t
  \eta(t) = \mtransport p {q(t)} \derivby t \mtransport {q(t)} p
  \eta(t)$.
\item A Hilbert covariant derivative
  \begin{equation*}
    \oDerivby t \alpha(t) = \frac12 \left(\mDerivby t \alpha(t) +
      \eDerivby t \alpha(t) \right) = 
    \dot\alpha(t) - \frac12 \expectat{q(t)}{\dot\alpha(t)} + \frac12
    \velocity q(t) \alpha(t) \ ,
  \end{equation*}
\end{enumerate}
corresponding to the canonical Riemannian (or Levi-Civita) connection on the Hilbert bundle.

\begin{remark}
We have used here a presentation based on one-dimensional
statistical models. From the differential geometry point of view is
more common to define covariant derivation on a vector field. We briefly comment about this issue below.

Given two smooth section $X, Y$ of the statistical bundle, that
is two differentiable mappings $X, Y \colon \maxexpat \mu \to \reals^N$, such
that for all $q$ it holds $\expectat q {X(q)} = \expectat q {Y(q)} =
0$, the covariant derivative is defined by
\begin{equation*}
  D_Y X (q) = \left. \Derivby t X(q(t))\right|_{t=0} \quad \text{for
    $q(0)=q$ and $\velocity q(0) = Y(q)$}. 
\end{equation*}

A detailed discussion of the geometry associated to our setting should
include, for example, the computation of the Christoffel coefficients
and the curvature of each of the three connections we have
introduced. Some of these computations are not really relevant for our main
goal, that is, the foundations of the mechanics of the statistical
bundle. Others are probably useful and interesting.

As an example, let us verify that the Hilbert connection defined above is the unique Levi-Civita connection. We shall check that the connection is torsion-free, that is $D_Y X - D_X Y = [X,Y]$, where $[X,Y]$ indicates the commutator of vector fields in the Hilbert bundle modeled on the Riemannian connection of the positive sphere.

We have, for each $\velocity q = Y(q)$ and $q(0) = q$, that
\begin{multline*}
  D_YX(q(t)) = \derivby t X(q(t)) - \frac12 \expectat {q(t)} {\derivby
    t X(q(t))} + \frac12 \velocity q(t) X(q(t)) = \\
  DX(q(t)) [\dot q(t)] - \frac12 \expectat {q(t)} {DX(q(t)) [\dot q(t)]} +
  \frac12 X(q(t)) Y(q(t)) = \\
    q(t)\, DX(q(t)) [Y(q(t))] - \frac12 \expectat {q(t)} {q(t) DX(q(t)) [Y(q(t))]} +
  \frac12 X(q(t)) Y(q(t)) \ .
\end{multline*}

The form of the Hilbert covariant derivative in terms of ordinary
derivatives of fields is
\begin{equation*}
  D_X(q) = q\,DX(q)[Y(q)] - \frac12 \expectat q {q\,DX(q)[Y(q)]} + \frac12
  X(q)Y(q) \ .
\end{equation*}

It follows that the bracket is
\begin{equation*}
  [X,Y](q) = D_YX(q) - D_XY(q) = q\,DX(q)[Y(q)] - q\,DY(q)[X(q)] \ . 
\end{equation*}
In fact, the expectation term is zero because $\expectat q {[X,Y](q)}
= 0$.
\end{remark}

We define the \emph{second statistical bundle} to be
\begin{equation}
\exptwobundleat \mu = \setof{(q,w_1,w_2,w_3)}{(q \in \maxexpat \mu,w_1,w_2,w_3 \in \expfiberat q \mu} \ ,
\end{equation}
with charts centered at each $p \in \maxexpat \mu$ defined by
\begin{equation}
s_p(q,w_1,w_2,w_3) = \left(s_p(q),\etransport q p w_1,\etransport q p w_2,\etransport q p w_3\right) \ .  
\end{equation}

The second bundle is an expression of the tangent bundle of the
exponential bundle. For each curve $t \mapsto \gamma(t) =
(q(t),w(t))$ in the statistical bundle, we define its \emph{velocity
  at $t$} to be
\begin{equation}
  \velocity \gamma (t) =  \left(q(t),w(t),\velocity q(t),\Derivby t w(t)\right) \ ,
\end{equation}
because $t \mapsto \velocity \gamma(t)$ is a curve in the second
statistical bundle and that its expression in the chart at $p$ has the
last two components equal to the values given in \cref{eq:deriv1} and
\cref{eq:deriv2}, respectively. The corresponding notion of gradient
will be discussed in the next section. 

In particular, for each smooth curve $t \mapsto q(t)$, the velocity of
its  lift $t \mapsto \gamma(t) = (q(t),\velocity
q(t))$ is 
\begin{equation}
\velocity \chi(t) = \left(q(t),\velocity q(t), \velocity q(t), \acceleration q(t)\right) \ ,  
\end{equation}
where the \emph{acceleration} $\acceleration q(t)$ at $t$ is 
\begin{equation}\label{eq:acceleration}
\acceleration q(t) = \Derivby t \velocity q(t) = \derivby t \frac{\dot q(t)}{q(t)}   - \expectat {q(t)} {\derivby t \frac{\dot q(t)}{q(t)}} = \frac {\ddot q(t)}{q(t)} - \Big(\velocity q(t)^2 - \expectat {q(t)} {\velocity q(t)^2}\Big) \ .\end{equation}
Notice that the computations above are performed in the embedding
space.

The acceleration has been defined using the transports. Indeed, the connection
here is defined by the transports $\etransport p q$, an approach that seems natural from the probabilistic point of view,
cf. \cite{gibilisco|pistone:98}. The
non-parametric approach to Information Geometry allows to define naturally a dual
transport, hence the dual connection of
\cite{amari|nagaoka:2000}.

The acceleration defined above has the one-dimensional exponential
families as (differential) geodesics. Every exponential (Gibbs) curve
$t \mapsto q(t) = e_p(t u)$ has velocity
$\velocity q(t) = u - D K(t u)[u]$, so that  the acceleration is
$\acceleration q(t) = 0$. Conversely, if one writes $v(t) = \log q(t)$, then
\begin{equation*}
 0 = \acceleration q(t) = \ddot v(t) + \expectat {q(t)} {\ddot v(t)} \ ,
\end{equation*}
so that $v(x;t) = t v(x) + c(t)$.

\begin{example}\label{ex:chi-trick}
Let us discuss a representation of the acceleration that does not
involve the construction of a second-order bundle. Consider the curve $t \mapsto q(t)$ and its lift $t \mapsto
(q(t),\velocity q(t)) \in \expbundleat \mu$. From the retraction $(q,\velocity q) \rightarrow (q,\chi)$, we can define a new curve 
\begin{equation*}
  t \mapsto \chi(t) = e_{q(t)}(\velocity q(t)) = \euler^{\velocity
    q(t) - K_{q(t)}(\velocity q(t))} \cdot q(t) \ ,
\end{equation*}
such that $\velocity q(t) = s_{q(t)}(\chi(t))$.

Let us compute the velocity of $\chi$.
\begin{multline*}
  \velocity \chi(t) = \derivby t  \log \chi(t) = \derivby t \Big (   \velocity q(t) - K_{q(t)}(\velocity q(t)) + \log q(t) \Big)
  = \\
 \acceleration q(t) +  \expectat {q(t)} {\derivby t \velocity
   q(t)} - \derivby t K_{q(t)}(\velocity q(t)) + \velocity q(t)=
  \acceleration q(t) + \velocity q(t) + c(t) \ .
\end{multline*}
where $c(t)$ is a scalar. That is, $\velocity \chi(t)$ and
$\acceleration q(t) + \velocity q(t)$ differ by a scalar, in particular,
%
\begin{equation*}
  \expectat {q(t)} {\velocity \chi(t)} + \expectat {\chi(t)}
  {\acceleration q(t) + \velocity q(t)} = 0 \ .
\end{equation*}

In conclusion, the following representation of the acceleration $\acceleration q$ in terms of the velocities $\velocity \chi$ and $\velocity q$ holds true:
\begin{equation}\label{eq:acctovel}
   \acceleration q = \etransport {\chi} {q} 
   \velocity \chi - \velocity q \quad \text{and} \quad  \velocity \chi = \etransport {q} {\chi} \left(\acceleration q + \velocity q\right)  \ .
  \end{equation}

  In a chart centered at $p$, we have $q(t) = e_p(u(t))$,
  $\velocity q(t) = \etransport p {q(t)} {\dot u(t)}$, so that
  $\chi(t)/p \propto \dot u(t) + u(t)$. In particular, in the case of
  an exponential model, $u(t) = tu$ and $\chi(t) = q(t+1)$, a property
  that is equivalent to the geodesic property
  $\acceleration q(t) = 0$.
\end{example}

We can also define other types of acceleration. In fact, we have three
different interpretation of the lifted curve, namely, we can
consider $t \mapsto (q(t),\velocity q(t))$ as a curve in the
statistical bundle 
$\expbundleat \mu$, or, a curve in the dual bundle $\mixbundleat
\mu$, or, a curve in the Hilbert bundle. Each of these frameworks
provides a different derivation, hence, a different acceleration.

We have already defined the \emph{exponential acceleration} $\eacc q(t) = \acceleration
q(t)$. Further, we can define the \emph{mixture acceleration} as
\begin{equation}\label{eq:macc}
\macc q(t) = \mDerivby t \velocity q(t) = \mtransport p {q(t)} \derivby t  \mtransport {q(t)} p
\velocity q(t) = \ddot q(t)/q(t)
\end{equation}
and the \emph{Riemannian
  acceleration} by
\begin{equation}\label{eq:0acc}
  \acc q(t) = \frac12 \left(\eacc q(t) + \macc q(t)\right) = \frac {\ddot q(t)}{q(t)} - \frac12\left(\left(\frac {\dot q(t)}{q(t)}\right)^2 - \expectat {q(t)}{\left(\frac {\dot q(t)}{q(t)}\right)^2}\right) \ ,
\end{equation}
In the review papers by \cite{pistone:2013GSI,pistone:2020-Pescara}, the
various accelerations are used to derive the relevant Taylor
formul\ae \, and the relevant Hessians. Moreover, it is shown that
the Riemannian acceleration can be derived using a family of isometric transport on the Hilbert bundle. Here, we will be mostly interested in the
mechanical interpretation of the acceleration. 

\section{Natural Gradient}
\label{sec:natural-gradient}

In this section we generalize the (non-parametric) natural gradient to the statistical bundles. Let us first recall the definition we are going to generalize. Given a 
  scalar field $F \colon \maxexpat \mu \to \reals$ the natural
  gradient is the section $q \mapsto \Grad F(q)$ of the dual 
  bundle $\mixbundleat \mu$ such that for all smooth curve
  $t \mapsto q(t) \in \maxexpat \mu$ it holds
\begin{equation}\label{eq:naturalgradient}
  \derivby t F(q(t)) = \scalarat {q(t)} {\Grad F(q(t))}{\velocity
    q(t)} \ .
\end{equation}

The natural gradient can be computed in some cases without recourse to the computation in charts, for example, 
\begin{multline}\label{eq:grad-entropy}
    \derivby t \entropyof{q(t)} = -\derivby t \expectat 1 {q(t) \log q(t)} = -\expectat 1 {\dot q(t)(\log q(t) + 1)} = \\  - \expectat {q(t)} {\log q(t) \velocity q(t)} = \scalarat {q(t)} {- \log q(t) - \entropyof {q(t)}}{\velocity q(t)} \ .
\end{multline}
In general, the natural gradient could be expressed in charts as a function of the ordinary gradient $\nabla$ as follows. In the generic chart at $p$, with $q = e_p(u)$ and $F(q) = F_p(u)$, it holds 
\begin{multline}
  \label{eq:naturalgradientinchart}
  \scalarat {q(t)} {\Grad F(q(t))}{\velocity
    q(t)} = \derivby t F(q(t)) = \derivby t F_p(u(t)) = D F_p(u(t))[\dot u(t)] = \\
  D F_p(u(t))[\etransport {q(t)} p \velocity q(t)] = \scalarat {p} {p^{-1} \nabla F_p(u(t)) }{\etransport {q(t)} p \velocity q(t)} = \\ \scalarat {q(t)} {\mtransport p {q(t)} p^{-1} \nabla F_p(u(t)) }{\velocity q(t)} = 
  \scalarat {q(t)} {q(t)^{-1} \nabla F_p(u(t)) }{\velocity q(t)} = \\ \scalarat {q(t)} {q(t)^{-1} \nabla F_p(u(t) ) - \expectat {q(t)} {q(t)^{-1} \nabla F_p(u(t) )} }{\velocity q(t)}\ .
\end{multline}

\

Therefore, the natural gradient is defined as
\begin{equation*}
    \Grad F(q) = q^{-1} \nabla F_p(u(t) ) - \expectat {q} {q^{-1} \nabla F_p(u(t))} 
\end{equation*} 
We use here the name of natural gradient for a computation which does not involve the Fisher matrix because of our choice of the inner product. The push forward of our definition to the tangent bundle of the simplex with the Fisher metric would indeed map our definition to the Riemannian one. \\



We are going to generalize the computation of the gradient to other
cases are of interest, namely, the \emph{Lagrangian} function, or Lagrangian
field, defined on the exponential bundle
$\expbundleat \mu$, and the \emph{Hamiltonian} function, or Hamiltonian field, defined on the dual bundle $\mixbundleat \mu$.

To include both cases, we derive below the generalization of natural
gradient to functions defined on the full statistical bundle
$\fullbundleat \mu$ and possibly depending on external
parameters. While this derivation is essentially trivial,
nevertheless we present here a full proof in order to introduce and
clarify the geometrical features of our presentation of the mechanics
of the open probability simplex in the next section.

In the statistical bundles the partial derivatives are not defined, but they
are defined in the trivialisations given by the affine
charts. Precisely, let be given a scalar field $F \colon \fullbundleat
\mu \times \mathcal D \to
\reals$, $\mathcal D$ a domain of $\reals^k$,  and a generic smooth
curve
\begin{equation*}
  t \mapsto (q(t),\eta(t),w(t),c(t)) \in \fullbundleat \mu \times
  \mathcal D \ .
\end{equation*}

We want to write 
\begin{multline}\label{eq:totalderivative}
  \derivby t F\big(q(t),\eta(t),w(t),c(t)\big) = \\ \scalarat {q(t)} {
    \Grad F\big(q(t),\eta(t),w(t),c(t)\big) }{\velocity q(t)} + 
\scalarat
  {q(t)} {\Derivby t \eta(t)} { \Grad_\text{m}
    F\big(q(t),\eta(t),w(t),c(t)\big)} + \\   \scalarat {q(t)} {\Grad_\text{e} F\big(q(t),\eta(t),w(t),c(t)\big)}{\Derivby t w(t)}
  + \nabla F\big(q(t),\eta(t),w(t),c(t)\big) \cdot \dot c(t) \ ,
\end{multline}
where the four components of the gradient are
\begin{equation*}
\fullbundleat \mu \times \mathcal D \ni (q,\eta,w,c) \mapsto \begin{cases} (q,\Grad F\big(q,\eta,w,c\big))
  \in \mixfiberat q \mu \\
(q,\Grad_\text{m} F\big(q,\eta,w,c\big))
  \in \expfiberat q \mu \\  (q,\Grad_\text{e} F\big(q,\eta,w,c\big))
               \in \mixfiberat q \mu \\  (q,\nabla  F\big(q,\eta,w,c\big))
  \in \maxexpat \mu \times \reals^k
\end{cases}
\end{equation*}
\

Let us fix a reference density $p$ and express both the given function and the
generic curve in the chart at $p$. We can write the total derivative as
\begin{multline*}
  \derivby t F \big(q(t),\eta(t),w(t),c(t)\big) = \derivby t
   F_p(u(t),\zeta(t),v(t),c(t)) = \\
  D_1 F_p \big(u(t),\zeta(t),v(t),c(t)\big) \big[\dot u(t) \big] + D_2
  F_p \big(u(t),\zeta(t),v(t),c(t)\big) \big[\dot \zeta(t) \big] + \\
   D_3 F_p\big(u(t),\zeta(t),v(t),c(t)\big) \big[\dot v(t) \big] +
   D_4 F_p\big(u(t),\zeta(t),v(t),c(t)\big) \big[\dot c(t) \big] \ . 
\end{multline*}

In the equation above, $D_i$, with $i=1,\dots,4$ denotes the differential of $F_p$ with
respect to the $i$-th variable, which is intended to provide a linear
operator to be represented by the appropriate dual vector, that is,
the value of the proper gradient.

The last term does not require any comment and we can use the ordinary
Euclidean gradient:
\begin{equation*}
  D_4 F_p\big(u(t),\zeta(t),v(t),c(t)\big) \big[\dot c(t) \big] =
  \nabla F_p\big(u(t),\zeta(t),v(t),c(t)\big) \cdot \dot c(t) \ .
\end{equation*}

Let us consider together the second and the third term. This is a
computation of the fiber derivative and does not involve the
representation in chart. Given
$\alpha \in \mixfiberat p \mu$ and $\beta \in \expfiberat p \mu$, that
is, $(\alpha,\beta) \in \fullfiberat p \mu$, we have
\begin{multline*}
  D_2 F_p(u,\zeta,v,c) [\alpha] + D_3 F_p(u,\zeta,v,c) [\beta] =
  \left. \derivby t F_p(u,\zeta+t\alpha,w+t\beta,c) \right|_{t=0} = \\
  \left. \derivby t F(q,\eta +t \mtransport p q \alpha,v + t
    \etransport p q \beta,c) \right|_{t=0} = \mathbb F
  F(q,\eta,w,c)[(\mtransport p q \alpha,\etransport p q
  \beta)] = \\
  \scalarat q {\mtransport p q \alpha} {\Grad_\text{m} F(q,\eta,w,c)}
  + \scalarat q {\Grad_\text{e} F(q,\eta,w,c)} {\etransport p q \beta}
  \ ,
\end{multline*}
where $\mathbb F$ denotes the fiber derivative in $\fullfiberat q
\mu$, which is expressed, in turn, with the relevant gradients. The
notation is possibly confusing, but consider that the inner
product has is always $\mixfiberat q \mu$ first, followed by $\expfiberat
q \mu$ and that the subscript to the $\Grad$ symbol displays which
component of the full bundle is considered.

We have that
\begin{equation*}
  \Derivby t w(t) = \etransport p q(t) \dot v(t) \ , \quad \Derivby
  t \eta (t) = \mtransport p {q(t)} \dot \zeta(t) \ .
\end{equation*}
Putting together all results up now, we have proved that
\begin{multline*}
  \derivby t F\big(q(t), \eta(t),w(t),c(t)\big) =\\  D_1 F_p
                                                   \big(u(t),\zeta(t),v(t),c(t)\big)
                                                  \big[\etransport
                                                   {e_p(u(t))} p
                                                   \velocity q(t)  \big]  + \\
                                                 \scalarat
                                                   {q(t)} {\Derivby t \eta(t)} { \Grad_\text{m}
                                                   F\big(q(t),\eta(t),w(t),c(t)\big)} + \\  \scalarat {q(t)} {\Grad_\text{e} F\big(q(t),\eta(t),w(t),c(t)\big)}{\Derivby t w(t)}
                                                                                               +\\  \nabla F\big(q(t),\eta(t),w(t),c(t)\big) \cdot \dot c(t) \ ,
\end{multline*}

To identify the first term in the total derivative above, consider the
``constant'' case,
\begin{equation*}
  q(t) = e_p(u(t)), \quad \eta(t) =
  \mtransport p {e_p(u(t))} \zeta,  \quad w(t) = \etransport p {e_p(u(t))} v, \quad c(t) = c \ ,
\end{equation*}
so that the first term reduces to $D_1 F_p(u(t),\zeta,v,c) [\etransport {e_p(u(t))} p \velocity q(t) ]$. It
follows that the proper way to compute the first gradient is to
consider the function on $\maxexpat \mu$ defined by 
 \begin{equation*}
 q \mapsto F_{\zeta,v,c}(q) = F(q,\mtransport p q \zeta,\etransport p q v,c)
 \end{equation*}
which  has a natural gradient whose chart representation is precisely
that first term.

We state the results obtained above in the following formal statement.
\begin{proposition}\label{prop:totalderivative} 
The total derivative \cref{eq:totalderivative} holds true, where
\begin{enumerate}
\item \label{prop:totalderivative1} $\Grad F\big(q,\eta,w,c\big)$ is the natural gradient of
  \begin{equation*}
    q
  \mapsto F(q,\mtransport p q \zeta,\etransport p q v,c) \ ,
  \end{equation*}
that is, with the representation in $p$-chart
\begin{equation*}
  F_p(u,\zeta,w,c) = F(e_p(u),\mtransport p {e_p(u)} \zeta,\etransport p {e_p(u)} v,c) \ ,
\end{equation*}
it is defined by
\begin{equation*}
      \scalarat{q}{\Grad F(q,\zeta,w,c)}{\velocity q} = D_1F_p(u,\zeta,w,c) \left[\etransport{q}{p} \velocity q\right] \ , \quad (q,\velocity q) \in \expbundleat \mu \ ;
  \end{equation*}
\item \label{prop:totalderivative2} $\Grad_\text{m} F\big(q,\eta,w,c\big)$ and $\Grad_\text{e}
  F\big(q,\eta,w,c\big)$ are the fiber gradients;
\item $\nabla F\big(q,\eta,w,c\big)$ is the Euclidean gradient w.r.t.~the last variable.
\end{enumerate}
\end{proposition}

We have concluded the computation of the total derivative of a
parametric function of the full bundle. The special cases of the
Lagrangian and the Hamiltonian easily follows as a
specialization. Notice that the computation of the natural gradient \eqref{prop:totalderivative1} in Proposition
\ref{prop:totalderivative} is done by fixing the variables in the fibers to be translations of
fixed ones.

We are going to discuss the following examples: the \emph{quadratic
  Lagrangian} $L(q,w) = \frac12 \scalarat q w w$; the \emph{cumulant
  Lagrangian} $L(q,w) = K_q(w)$; \emph{conjugate cumulant Hamiltonian}
$H(q,\eta) = \expectat q {(1+\eta) \log (1 + \eta)}$. Detailed
computations of the relevant gradients are given in
\ref{sm:computations}.

\section{Mechanics of the Statistical Bundle}
\label{sec:mechanics}
If $q \colon [0,1] \ni t \mapsto q(t)$ is
a smooth curve in the exponential manifold $\maxexpat \mu$ and
$t \mapsto (q(t),\velocity q(t))$ is its lift to the
statistical bundle, the \emph{action integral} of the Lagrangian $L
\colon \expbundleat \mu \times [0,1] \to \reals$ is
  \begin{equation}
q \mapsto   A(q) = \int_{0}^{1} L(q(t),\velocity q(t),t) \ dt \ .
\end{equation}
In the exponential chart $s_p$ centered at $p$, the curve is $q(t) =
\euler^{u(t)-K_p(u(t))}\cdot p$, the coordinates of the lift are $s_p(q(t),\velocity q(t)) =
\left(u(t),\dot u(t)\right)$, and the expression of the Lagrangian is
\begin{equation*}
L(q(t),\velocity q(t),t) = L\left(e_p(u(t)),\etransport p {e_p(u(t))} \dot
u(t),t\right) = L_p(u(t),\dot u(t),t) \ . 
\end{equation*}
The expression of the action integral in coordinates is
\begin{equation}\label{eq:action-in-chart}
  u \mapsto A_p(u) = \int_{0}^{1} L_p(u(t),\dot u(t),t) \ dt \ , 
\end{equation}
and the Euler-Lagrange equation of \cref{eq:action-in-chart} is 
\begin{equation*}
   D_1 L_p(u(t),\dot u(t),t)[h] = \derivby t D_2 L_p(u(t),\dot u(t),t)[h] \ . \quad t \in [0,1] \ , h \in \expfiberat p \mu \ . 
\end{equation*}

By using Proposition \ref{prop:totalderivative}, we obtain the intrinsic
expression of the Euler-Lagrange equations in the statistical bundle.
\begin{proposition}[The Euler-Lagrange equation]
  If $q$ is an extremal of the action integral, then, with the
  notations of Proposition \ref{prop:totalderivative}, 
  \begin{equation}\label{eq:Euler-Lagrange}
  \Derivby t \Grad_\text{e} L(q(t),\velocity q(t),t) = \Grad L(q(t),\velocity q(t),t) \ .
\end{equation}
\end{proposition}

At each fixed density $q \in \maxexpat \mu$, and each time $t$, the
partial mapping
$\expfiberat q \mu \ni w \mapsto L_{q,t}(w) = L(q, w, t)$is defined on
the vector space $\expfiberat q \mu$, and its gradient mapping in the
duality of $\mixfiberat q \mu \times \expfiberat q \mu$ is
$w \mapsto \Grad_\text{e} L(q,w,t)$. The standard Legendre transform
argument provides the intrinsic form of the Hamilton equations under
the following assumption.

\begin{assumption}\label{h1}
  We will always restrict our attention to Lagrangians such that the
  fiber gradient mapping at $q$,
  $w \mapsto \eta = \Grad_\text{e} L_q(w)$ is a 1-to-1 mapping from
  $\expfiberat q \mu$ to $\mixfiberat q \mu$. In particular, this true
  when the partial mappings $w \mapsto L_q(w)$ are strictly convex for
  each $q$. 
\end{assumption}
In our finite dimensional context, this assumption is actually
equivalent to the assumption that the fiber gradient is a
diffeomorphism of the statistical bundles
$\Grad_{\text{2}} L \colon \expbundleat \mu \to \mixbundleat
\mu$. This is related to the properties of \emph{regularity} and
hyper-regularity, cf. \cite[\S~3.6]{abraham|marsden:1978}. The bilinear form
$\mixfiberat q \mu \times \expfiberat q \mu \ni (\eta,w) \mapsto
\scalarat q \eta w = \expectat q {\eta w}$ will always be written in
this order. The Legendre transform of $L_{q,t}$ is defined
for each $\eta \in \mixfiberat q \mu$ of the image of $\Grad_\text{e}
L(q, \cdot, t)$, so that the Hamiltonian is
\begin{equation}
  \label{eq:derived-Hamiltonian}
  H(q,\eta,t) = \scalarat q {\eta}{(\Grad_\text{e} L_{q,t})^{-1}(\eta)} - L(q,(\Grad_\text{e}
  L_{q,t})^{-1}(\eta)) \ .
\end{equation}
If $t \mapsto q(t)$ a solution of Euler-Lagrange
\cref{eq:Euler-Lagrange}, the curve
$t \mapsto \zeta(t) = (q(t),\eta(t))$ in $\mixbundleat \mu$, where
$\eta(t) = \Grad_\text{e} L(q(t),\velocity q(t),t)$ is the
\emph{momentum}. 
\begin{proposition}[The Hamilton equations]\label{prop:Hamiltonequation}
  When Assumption \ref{h1} holds, the momentum curve satisfies the \emph{Hamilton equations}, 
\begin{equation}
  \label{eq:Hamilton}
  \left\{\begin{aligned}
    \Derivby t \eta(t) &= - \Grad H(q(t),\eta(t),t) \\
     \velocity q(t) &= \Grad_\text{m} H(q(t),\eta(t),t) .
  \end{aligned}\right.
\end{equation}
Moreover,
\begin{equation}
  \label{eq:conservationH}
  \derivby t H(q(t),\eta(t),t) = \pderivby t H(q(t),\eta(t),t) \ .
\end{equation}
\end{proposition}
The two proposition above do not require here a proof because they are
standard results in classical mechanics. What is relevant here is the special intrinsic form which is obtained by the use of the covariant
derivatives and the gradients of the statistical bundles.

As simple examples, one can compute the Euler-Lagrange equation and
the Hamilton equation for the quadratic Lagrangian and the cumulant
Lagrangian, see \ref{sm:computations}.  A slightly more sophisticate
example is inspired by the standard \emph{free particle} Lagrangian,
where the role of the point particle is played by a probability
density as a point on the statistical manifold 
\cite{pistone:2018Lagrange}. The Lagrangian is written as a difference
of the quadratic form and a potential function given by the negative
of the entropy function $\entropyof {q(t)}$.  We keep the inertial
mass $m$ as a parameter,
\begin{equation*}L(q,w) = \frac m2 \scalarat q {w}{w} + \kappa
  \entropyof q \ , m, \kappa > 0, (q,w) \in \expbundleat \mu \ .
\end{equation*}
The first component of the natural gradient is easily computed and the natural gradient of the entropy is \cref{eq:grad-entropy}. The
Euler-Lagrangian equation is
\begin{equation}\label{eq:dynamic-example}
  m {\Derivby t \velocity q(t)} = \frac m 2\left(\velocity q(t)^2 -
     \expectat{q(t)}{\velocity q(t)^2}\right) - \kappa (\log q(t) + \entropyof {q(t)} \ ,
 \end{equation} 
which is Newton's law, written in terms of the mixture acceleration. 
We express this equation as a system of ordinary differential equations for $q$ and $\velocity q$ by writing $v(t) = \velocity q(t)$ and note
that $v(t) = \derivby t \log q(t)$. (see cf. \ref{odes}).
\begin{figure}
\begin{center}
\includegraphics[width=\textwidth]{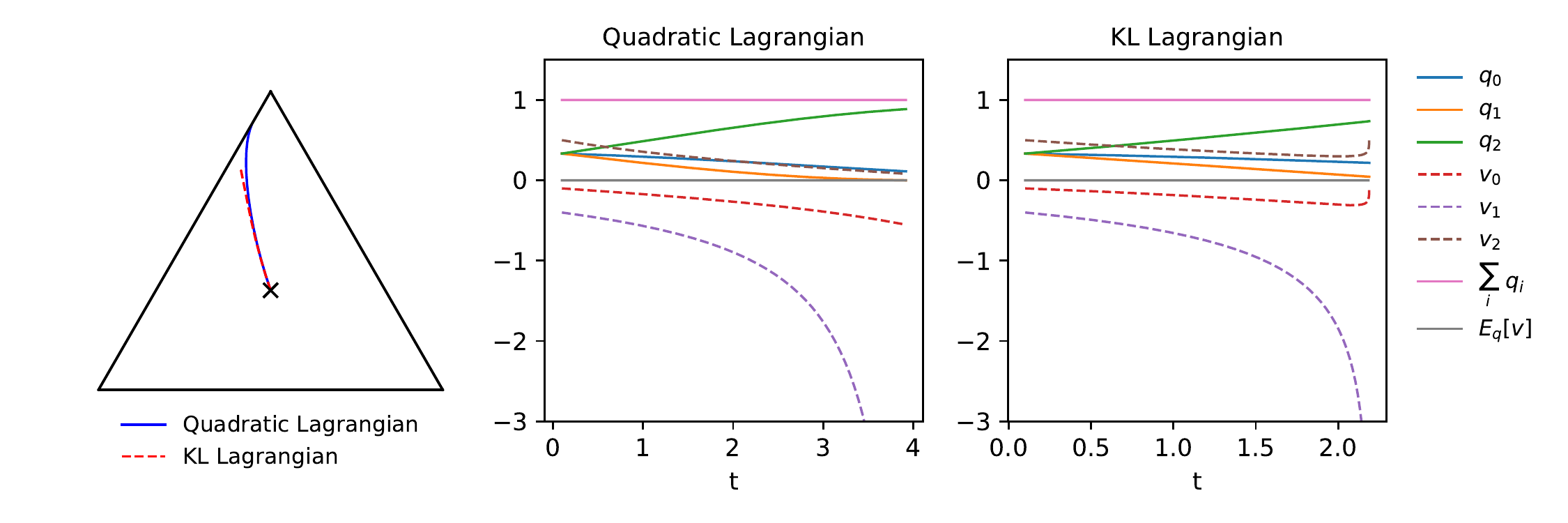}
\caption{Free-motion solutions on the simplex for the quadratic (blue) and KL (dashed red) Lagrangians. 
In the quadratic case, the geodesic motion  approaches the boundary of the simplex tangentially, where one component of the probability vanishes (while the associated velocity diverges). Similarly, the solution of the free KL Lagrangian flow moves toward the boundary of the simplex. In the latter case, however, the components of the score velocity diverge at the boundary, while the probability tends to it non tangentially. Both systems share the same initial conditions (black cross): $q_0 = (1/3, 1/3, 1/3), w_0 = (-0.1, -0.4,  0.5)$. 
All the systems of ordinary differential equations are solved numerically via Python ODEINT integrator.}\label{fig:free-motion}
\end{center}
\end{figure}
\begin{figure}
\begin{center}
\includegraphics[width=\textwidth]{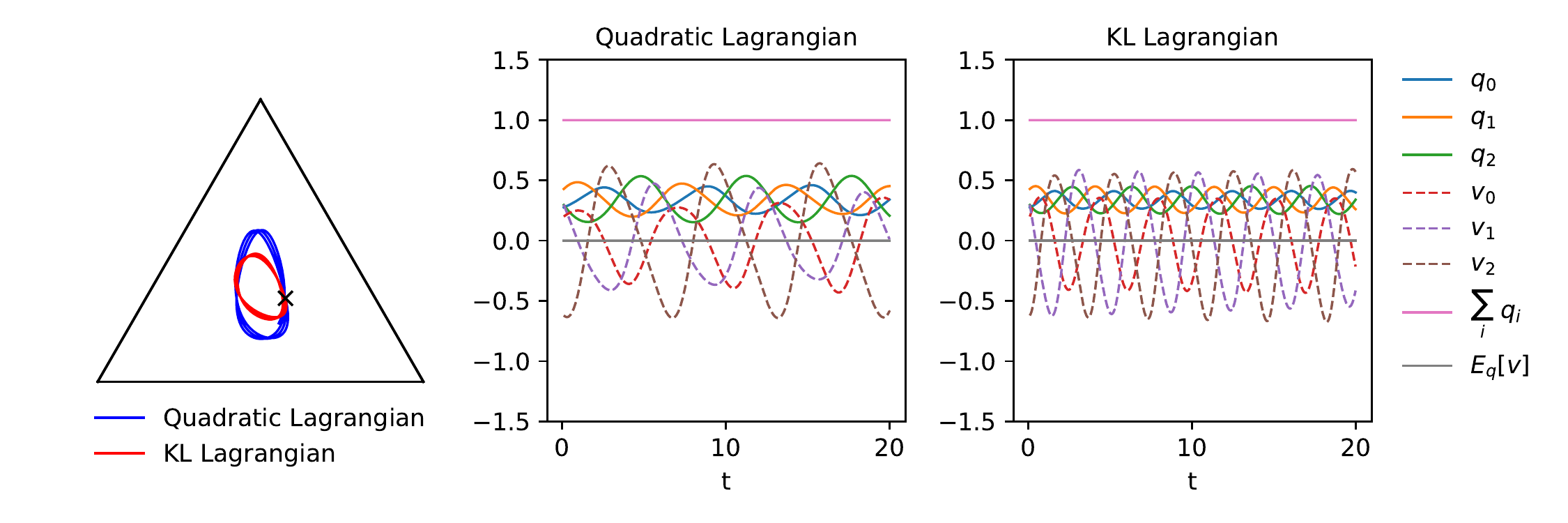}
\caption{Solutions of the E-L equations, for the quadratic (blue, cf.~\cref{eq:dynamic-example}) and the KL (red, cf.~\cref{eq:EL-2}, $a=b=1$) Lagrangians, describing the motion in a potential on the simplex. Both systems show the expected harmonic oscillating behavior, while generally displaying different trajectories for the same initial conditions (black cross).The explicit ODEs systems for the two cases are given in \ref{odes}.}\label{simp2}
\end{center}
\end{figure}

A non-quadratic, non-symmetric generalization of the kinetic energy on the simplex is realised via the Kullback-Leibler divergence. We introduce this case by first discussing the very idea of deducing a Lagrangian, or a Hamiltonian, from a divergence function.

A divergence is a smooth mapping
$D \colon \maxexpat \mu \times \maxexpat \mu \to \reals$, such that
for all $p,q \in \maxexpat \mu$ it holds $D(p,q) \geq 0$ and
$D(p,q) = 0$ if, and only if, $p=q$. Typically, a divergence is
not symmetric, and frequently the discussion involves both the
divergence and the so-called dual divergence $D^*(p,q) = D(q,p)$. Every divergence can be associated to a Lagrangian by the canonical
mapping:
\begin{equation} \label{eq:goffredo}
  \maxexpat \mu ^2 \ni (q,r) \mapsto (q,s_q(r)) = (q,w) \in
\expbundleat \mu \ ,
\end{equation}
where $r = \euler^{w - K_q(w)} \cdot q$, that is, $w = s_q(r)$. The inverse mapping is the \emph{retraction}
\begin{equation}
  \label{eq:EXP}
  \expbundleat \mu \ni (q,w) \mapsto (q,e_q(w)) = (q,r) \in \maxexpat
  \mu ^2 \ .
\end{equation}
As the curve $t \mapsto e_q(tw)$ has null exponential acceleration,
one could say that \cref{eq:EXP} defines the \emph{exponential} mapping of
the exponential connection, while \cref{eq:goffredo} defines the
so-called \emph{logarithmic} mapping. 

The expression in a chart centered at $p$ of the mapping
of \cref{eq:EXP} is affine:
\begin{equation}\label{eq:EXP-in-chart}
  \arraycolsep=1.2pt\def\arraystretch{1.8}
  \begin{array}{ccccccc}
   \expfiberat p \mu \times \expfiberat p \mu &\to& \maxexpat \mu \times
  \maxexpat \mu &\to& \expbundleat \mu &\to& \expfiberat p \mu \times \expfiberat p \mu \\ 
  (u,v) &\mapsto& (e_p(u),e_p(v)) &\mapsto& (e_p(u),s_{e_p(u)}(e_p(v)))
  &\mapsto& \left(u,\etransport p {e_p(u)}(v -u)\right) \ . 
\end{array}
\end{equation}
In conclusion, the correspondence above maps every divergence $D$ into
a \emph{divergence Lagrangian}, and conversely,
\begin{equation}\label{eq:goffredo-2}
    L(q,w) = D(q,e_q(w)) \ , \quad D(q,r) = L(q,s_q(r)) \ .
\end{equation}
Notice that, according to our assumptions on the divergence, the
divergence Lagrangian defined in \cref{eq:goffredo-2} is non-negative
and zero if, and only if, $w=0$.

Similarly, every divergence can be associated to a Hamiltonian function on the dual bundle by the
  canonical mapping:
\begin{equation} \label{eq:goffredo-star}
  \maxexpat \mu ^2 \ni (q,r) \mapsto (q,\eta_q(r)) = (q,\eta) \in
\mixbundleat \mu \ ,
\end{equation}
where $r = (1+\eta) \cdot q$. The inverse mapping is 
\begin{equation}
  \label{eq:EXP-star}
  \mixbundleat \mu \ni (q,\eta) \mapsto (q,(1+\eta) \cdot q) = (q,r) \in \maxexpat
  \mu ^2 \ .
\end{equation}

The canonical example is the Kullback-Leibler divergence $D(q,r)$, with cumulant Lagrangian
function $K_q(w)$. Accordingly,  the dual divergence $D(r,q) = \expectat r {\log \frac
      r q}$ is naturally associated to the Hamiltonian function
    \begin{equation*}
      H(q,\eta) = \expectat {(1+\eta) \cdot q} {\log \frac {(1+\eta)
          q} q} = \expectat q {(1+\eta) \log (1+\eta)} \ .
    \end{equation*}

\Cref{fig:free-motion,simp2}
provide an illustration of the free motion and the motion in a
potential for the quadratic and the Kullback-Leibler case. 

Motivated by our interest in optimization, we will discuss now, in detail, a family of parameterized Lagrangians of the following standard form,
\begin{equation}\label{eq:standard-lagrangian}
  L^{a,b,c}(q,w) = c(a^{-1} K_q(aw) - b f(q)) \ ,
\end{equation}
where $a,b,c > 0$ and $f$ is a scalar field on $\maxexpat \mu$ as, for
example, the negative entropy previously introduced. The Lagrangian above is parameterized in such a way that
\begin{displaymath}
\lim_{a \to 0} L^{a,b,b^{-1}}(q,w) = b^{-1}(D K_q(0)[w] - b f(q)) = -
f(q)\ .
\end{displaymath}
The cumulant term is the scaled Lagrangian $(q,w) \mapsto
a^{-1} K_q(aw)$ whose divergence function in terms of $q$ and $r = e_q(w)$ is
\begin{equation}\label{eq:ren}
  \frac1a \log \expectat q {\expof{a \left(\log \frac r q - \expectat q {\log
          \frac r q}\right)}} = \frac1a \log \expectat q {\left(\frac r q\right)^a}
 + D(q,r) \end{equation}
and the limit for $a \to 0$ is null. Notice that the first term on the rhs of \eqref{eq:ren} is proportional to the $a$-R\'enyi divergence of $r$ from $q$ \cite{renyi1961}. 

Here, the constant $a$ is intended to introduce a mass effect in the
model in such a way that $a=0$ implies that the Lagrangian lost any
dependence on the velocity $w$. We could talk also of an
inertia of the system. Typically, the notion of \emph{inertia}
describes the resistance of any physical object to any change in its
velocity. In our statistical setting, the dynamics of a state along
some direction in the manifold can be
interpreted as the result of the balance of a \emph{gain} of motion,
determined from the descent along some potential function (payoff),
against the \emph{cost} of motion to changes along a given direction from the given
state. In this sense, the Lagrangian vector field on the statistical manifold consistently minimizes the action of the
difference of a divergence and a potential function. 

From an optimization viewpoint, our variational problem corresponds to the minimization of an objective function, the potential, with a proximity constraint enforced via the kinetic energy term. The kinetic energy acts as a \emph{regulariser} for the velocities, leading to faster converging and more stable optimization algorithms \cite{betancourt2017conceptual,frana2019conformal,franca2020dissipative}. 

\bigskip
  
Let us derive the Hamiltonian of our standard Lagrangian in
  \cref{eq:standard-lagrangian}. If $f$ is a real function with convex
  conjugate $f^*(\eta) = \sup_w \scalarof \eta w - f(w)$, then
  $g(w) = c(a^{-1} f(aw) - b)$ defines a new function whose convex
  conjugate is $g^*(\eta) = c(a^{-1} f^*(c^{-1} \eta) + b)$, where the conjugate momentum $\eta=\Grad_{\text{e}} L^{a,b,c}(q,w)$ is now a function of the parameters. In the
  case of a convex function, the Legendre transform coincides with the
  convex conjugate on the interior of the proper domain. In our case,
  $f^*(\eta) = \expectat q {(1+\eta)\log(1+\eta)}$, $\eta > -1$, and
  $b = bf(q)$, so that
  
    \begin{equation}\label{eq:standard-hamiltonian}
      H^{a,b,c}(q,\eta) = c\left(a^{-1} \expectat q
    {(1+c^{-1}\eta)\log(1+c^{-1}\eta)} + bf(q)\right) \ .  
\end{equation}
\

As $\lim_{c \to \infty} c \expectat q
    {(1+c^{-1}\eta)\log(1+c^{-1}\eta)} = 0$, we have   
$\lim _{c \to \infty} H^{a,c^{-1},c}(q,\eta) \to  f(q)$.
\bigskip

We now proceed to compute the relevant natural gradients with
Proposition \ref{prop:totalderivative}. By computing the
total differential on the curve $t \mapsto (q(t),aw(t))$, we get 
\begin{multline*}
  \derivby t K_{q(t)}(aw(t)) = \\
  \scalarat {q(t)} {\left(\frac{e_{q(t)}(aw(t))}{q(t)} - 1\right) -
    aw(t)}{\velocity q(t)}+\scalarat {q(t)}
  {\left(\frac{e_{q(t)}(aw(t))}{q(t)} - 1\right)} {\Derivby t aw(t)} \ .
\end{multline*}

The gradients of the Lagrangian \eqref{eq:standard-lagrangian} are
\begin{equation}\label{eq:grad-standard-lagrangian}
  \begin{aligned}
     \Grad L^{a,b,c}(q,w) &= c\left(a^{-1}\left(\frac{e_q(aw)}{q} - 1\right) -
w - b \Grad f(q)\right) \\
    \Grad_{\text{e}} L^{a,b,c}(q,w) &= c\left(\frac{e_q(aw)}{q} - 1\right)
  \end{aligned}
\end{equation}
Now the limit cases are controlled by $\lim_{a\to0}
\left(\frac{e_q(aw)}{q} - 1\right) = 0$ and $\lim_{a\to0}
a^{-1}\left(\frac{e_q(aw)}{q} - 1\right) = w$, so that,
\begin{equation*}
  \lim_{a\to0} \Grad L^{a,b,b^{1}}(q,w) = - \Grad f(q) \quad \text{and}
  \quad \lim_{a\to0} \Grad_{\text{e}} L^{a,b,c}(q,w) = 0 \ .
\end{equation*}
\

The Euler-Lagrange equation is

\begin{equation}\label{eq:EL-0}
  \Derivby t \left(\frac{e_q(a\velocity q(t))}{q(t)} - 1\right) =
  a^{-1}\left(\frac{e_q(a\velocity q(t))}{q(t)} - 1\right) -
\velocity q(t) - b \Grad f(q(t)) \ .
\end{equation}
\

Let us compute the covariant time-derivative on the lhs using the trick of
\cref{ex:chi-trick}. We write $\chi_a(t) = e_q(a\velocity q(t))$ and recall we are using the mixture covariant derivative for $(\chi_a(t)/q(t)-1) \in \mixfiberat q \mu$. Then the left-hand side of the Euler-Lagrange equation becomes

\begin{equation}\label{eq:Dbyt-chi}
  \Derivby t \left(\frac{\chi_a(t)}{q(t)}-1\right) = q(t)^{-1} \derivby t (\chi_a(t) - q(t)) = \frac{\chi_a(t)}{q(t)}
  \velocity \chi_a(t) - \velocity q(t) = \mtransport {\chi_a(t)}{q(t)} \velocity \chi_a(t) - \velocity q(t) \ ,
\end{equation}

where
\begin{equation*}
  \velocity \chi_a(t) = \derivby t \log \chi_a(t) = \derivby t a
  \velocity q(t) - \derivby t K_{q(t)}(a\velocity q(t)) + \velocity q(t) \ ,
\end{equation*}
which in turn implies
\begin{equation*}
  \velocity \chi_a(t) = \etransport {q(t)}{\chi_a(t)} (a \acceleration q(t) + \velocity q(t)) \ .
\end{equation*}
so that
\begin{equation*}
  \Derivby t \left(\frac{\chi_a(t)} {q(t)} - 1\right) = \mtransport
  {\chi_a(t)}{q(t)} \etransport {q(t)}{\chi_a(t)} (a \acceleration
  q(t) + \velocity q(t))- \velocity q(t) \ .
\end{equation*}

The Euler-Lagrange equation becomes an equation in $q$, $\velocity q$, $\acceleration q$, 
\begin{equation*}
\mtransport
  {e_{q(t)}(a\velocity q(t))}{q(t)} \etransport {q(t)}{e_{q(t)}(a\velocity q(t))} (a \acceleration
  q(t) + \velocity q(t)) =  a^{-1}\left(\frac{e_q(a\velocity
      q(t))}{q(t)} - 1\right) - b \Grad f(q(t)) \ ,
\end{equation*}
or, moving the transports to the right-hand side,
\begin{equation}\label{eq:EL-1}
  a \acceleration
  q(t) + \velocity q(t) =  \etransport {e_{q(t)}(a\velocity q(t))} {q(t)}
  \mtransport {q(t)}
 {e_{q(t)}(a\velocity q(t))} \left(a^{-1}\left(\frac{e_{q(t)}(a\velocity
      q(t))}{q(t)} - 1\right) - b \Grad f(q(t))\right) \ ,
\end{equation}
where the right-hand side could be rewritten with
$\mtransport q \chi (\chi/q - 1) = 1 - q/\chi$. Notice that the limit
form as $a \to 0$ is $\Grad f(q(t)) = 0$. For example, if
$f(q) = - \entropyof q$, then $q(t) = 1$.
\bigskip

A similar argument applies to the computation of the gradients of the
Hamiltonian \eqref{eq:standard-hamiltonian}. The variation of $H(q,\eta) = \expectat q
{(1+\eta)\log(1+\eta)}$ on the curve $t \mapsto
(q(t),c^{-1}\eta(t))$ is
\begin{multline*}
  \derivby t H(q(t),c^{-1}\eta(t)) = 
  \scalarat {q(t)} {\log(1+c^{-1}\eta) - \expectat q
    {\log(1+c^{-1}\eta)} - c^{-1}\eta}{\velocity q(t)}+\\ \scalarat {q(t)}
   {\Derivby t c^{-1} \eta(t)} {\log(1+c^{-1}\eta) - \expectat q {\log(1+c^{-1}\eta)}} \ .
\end{multline*}
Substitution gives
\begin{gather*}
    \Grad H^{a,b,c}(q,\eta) = c\left(a^{-1}\left(\log(1+c^{-1}\eta) - \expectat q
      {\log(1+c^{-1} \eta)}\right) - c^{-1} \eta - b \Grad f(q)\right)\\
    \Grad_{\text{m}} H^{a,b,c}(q,\eta) = a^{-1}\left(\log(1+c^{-1}\eta) - \expectat q {\log(1+c^{-1}\eta)}\right)\\
   \end{gather*}
The Hamilton equations are
\begin{equation}\label{eq:H-1}
  \left\{
    \begin{aligned}
      \Derivby t \eta(t) &= - c\left(a^{-1}\left(\log(1+c^{-1}\eta) - \expectat q
      {\log(1+c^{-1} \eta)}\right) - c^{-1} \eta - b \Grad f(q)\right)
  \\
\velocity q(t) &=  a^{-1}\left(\log(1+c^{-1}\eta) - \expectat q {\log(1+c^{-1}\eta)}\right)   \end{aligned}
\right. \ .
\end{equation}

\begin{remark}
There is a way, other than the Hamilton equations, to write a
first-order system of differential equations equivalent to the
second-order Euler-Lagrange \cref{eq:EL-1}. We have found in \cref{eq:Dbyt-chi} that
the Euler-Lagrange \cref{eq:EL-0} can be written as
\begin{equation*}
  \frac {\chi_a(t)}{q(t)} \velocity \chi_a(t) - \velocity q(t) = a^{-1} \left(\frac
    {\chi_a(t)}{q(t)} - 1\right) - \velocity q(t) - b \Grad f(q(t)) \ ,  
\end{equation*}
which simplifies to
\begin{equation*}
  \chi_a(t) \velocity \chi_a(t) = a^{-1} (\chi_a(t) - q(t)) - b q(t)
  \Grad f(q(t)) \ ,
\end{equation*}
and, in turn, provides a remarkably simple system of replicator
equations,
\begin{equation}\label{eq:REP-1}
  \left\{
    \begin{aligned}
      \dot \chi_a(t) &= a^{-1}(\chi_a(t) - q(t)) - b q(t) \Grad f(q(t)) \\
      \dot q(t) &= q(t) \left(\log \frac {\chi_a(t)} {q(t)} - \expectat q {\log \frac {\chi_a(t)}
        {q(t)} }\right) 
    \end{aligned} \right. \ .
\end{equation}
Notice that the vector field is null if, and only if, $\chi_a = q$ and
$\Grad f(q) = 0$.
\end{remark}


We proceed now to express the equations we have obtained in the
ordinary Euclidean space. After computing the transports in the
right-hand side, the Euler-Lagrange \cref{eq:EL-1} becomes
\begin{multline}\label{eq:EL-2}
  a \acceleration
  q(t) + \velocity q(t) =  a^{-1}\left(\euler^{-a\velocity q(t) +
      K_{q(t)}(a \velocity q(t))} - \expectat {q(t)} {\euler^{-a\velocity q(t) +
      K_{q(t)}(a \velocity q(t))}}\right)
- \\
b \left(\euler^{-a\velocity q(t) +
      K_{q(t)}(a \velocity q(t))} \Grad f(q(t)) - \expectat {q(t)} {\euler^{-a\velocity q(t) +
        K_{q(t)}(a \velocity q(t))} \Grad f(q(t)) }\right) \ .
\end{multline}
Notice the common constant factor
\begin{equation*}
  \euler^{K_{q(t)}(a \velocity q(t))} = \expectat {q(t)}
      {\euler^{a\velocity q(t)}}
\end{equation*}
in each term.

There are many ways to rewrite \cref{eq:EL-2} as a system of ordinary differential equations in $\reals^{2N}$. An immediate option is to introduce the variables $q$ and $v =
\velocity q$, in which case the solution will stay in
the Grassmannian manifold $\sum_x q(x) v(x) = 0$. 
It holds $\derivby t q(t) = q(t) v(t)$.  The acceleration is
\begin{equation*}
\acceleration q(t) = \derivby t v(t) - \expectat {q(t)} {\derivby t
  v(t)} = \dot v(t) + \expectat {q(t)} {v(t)^2} \ .
\end{equation*}
An explicit expression for \cref{eq:EL-2} as a system of $2N$
ordinary differential equations is given in \cref{KLode}. The solution of the ODEs system is plot in \cref{simp2}.


The Hamiltonian equations \eqref{eq:H-1} form a differential system in
the two variables $q,\eta$ in the dual statistical bundle, which is an
open subset of the Grassmanian manifold. The solution curve and its
derivatives can be expressed in the global space in which the dual
bundle is embedded by observing that
$t \mapsto (q(t),\eta(t)) \in \mixbundleat \mu \subset \reals^\Omega
\times \reals^\Omega$ and using 
$\Derivby t \eta(t) = \frac{\dot q(t)}{q(t)} \eta(t) + \dot \eta(t)$,
$\velocity q(t) = \frac{\dot q(t)} {q(t)}$.

\section{Application to Accelerated Optimization}\label{sec:jordan}

As a final example of statistical Lagrangian dynamics, we consider the case of a damped mass-spring system on the probability space, defined via a time-dependent parameterized KL Lagrangian. 
This choice is motivated by a series of recent interesting results in optimization, where a time-dependent family of so-called \emph{Bregman Lagrangians} \cite{WibisonoE7351} is introduced to derive a variational approach to accelerated optimization methods.
 
 While the geometric setting in \cite{WibisonoE7351} is a generic Hessian manifold over a convex set in $\reals^d$, our goal is to reproduce such a derivation on the statistical bundle, as to provide a first consistent description of accelerated optimization on the dually-flat geometry of the exponential manifold. Recent related work on the accelerated gradient flow for probability distributions can be found in \cite{pmlr-v97-taghvaei19a}, \cite{wang2020accelerated}, and some relevant references therein. 
 
\subsection{Damped KL Lagrangian}

On the statistical bundle, let us consider a damped Lagrangian given by the difference of time-scaled KL divergence and potential function, multiplied by an overall time-dependent damping factor,
  \begin{equation}
 \label{eq:breglagr}
 \begin{aligned}
   L(q,w,t) &=  \euler^{\gamma_t} \left(
     \euler^{\alpha_t} D(q, e_{q}(\euler^{-\alpha_t}w)) -
     \euler^{\alpha_t+\beta_t} f(q) \right) \\
    &= \euler^{\alpha_t+\gamma_t}\,\left( \expectat{q}{\logof{\frac{1}{\expof{\euler^{-\alpha_tw}-K_q(\euler^{-\alpha_t}w)}}}}- \euler^{\beta_t} f(q) \right)\\
    &= \euler^{\alpha_t+\gamma_t}\,\left(K_q(\euler^{-\alpha_t}w)- \euler^{\beta_t} f(q) \right) \ .
  \end{aligned}
\end{equation}

For each fixed $t$, the time-dependent Lagrangian above is an instance of the standard Lagrangian \cref{eq:standard-lagrangian} with
\begin{equation*}
  a = \euler^{-\alpha_t} \ , \quad b = \euler^{\alpha_t+\beta_t} \ , \quad c =
  \euler^{\gamma_t} \ , 
\end{equation*}
which reproduces, on the statistical bundle, the time-dependent family of Bregman Lagrangians proposed in \cite{WibisonoE7351}. 
As in \cite{WibisonoE7351}, we assume $\alpha_t, \beta_t, \gamma_t: I \to \mathbb{R}$  to be continuously differentiable functions of time. The overall damping factor $\gamma_t$ is responsible for the dissipative behaviour of the Lagrangian system; $\beta_t$ provides the potential $f$ with an explicit time dependence; finally, $\alpha_t$ defines a scaling in time of the score velocity. 

In our setting, the scaling of the score is associated to a time-dependent lift to the statistical bundle. In the  exponential map, we consider a time-dependent scaling of the shift vector, such that $\chi=e_q(\euler^{-\alpha_t}w)$ and $s_p(\chi)=u+\euler^{-\alpha_t} v \in \expfiberat p \mu$, with $\alpha_t: I \to \mathbb{R}$ smooth, $I \subset \mathbb{R} $ open time interval.
With this choice the KL divergence reads 
\begin{equation*}
D \colon I \times \expbundleat \mu \ni (q,w,t) \mapsto D(q,e_q(\euler^{-\alpha_t}w)) \in \reals \ .
\end{equation*}

The overall scaling by the inverse factor  $\euler^{\alpha_{t}}$ makes the divergence closed under time-dilation and leads to a time-reparameterization invariant \cite{souriau:1970, 10.1007/978-3-030-80209-7_80, Chirco:2015bps} action 
 
 \begin{multline*}
     A(q(\tau), t(\tau)) = \int_{0}^{1} \dot{\tau}^{-1}\, d \tau\, \euler^{\alpha_{t(\tau)}}\, \euler^{\gamma_{t(\tau)}}\,\left[ D(q(t(\tau)),e_q(\euler^{-\alpha_{t(\tau)}}\velocity q(t(\tau))\,\dot{\tau} ))- \euler^{\beta_{t(\tau)}} f(q)\right]\\ = \int_{0}^{1} d \tau\, \euler^{\tilde{\alpha}_{\tau}}\,\euler^{\gamma_\tau}\, \left[ D(q(\tau), e_q(\euler^{-\tilde{\alpha}_{\tau}}\velocity q(\tau)) )- \euler^{\beta_\tau}\,f(q) \right]\ ,
 \end{multline*}
 where we set $\tilde{\alpha}_\tau=\alpha_t- \logof{\dot{\tau}}$.
 
It follows directly from \cref{eq:standard-hamiltonian} that the Hamiltonian is
\begin{equation}\label{eq:dampham}
  H(q,\eta,t) = \euler^{\alpha_t+\gamma_t} \left(\expectat q
    {(1+\euler^{-\gamma_t}\eta)\log(1+\euler^{-\gamma_t}\eta)} +
    \euler^{\beta_t} f(q)\right) \ . 
\end{equation}

The gradients of the Lagrangian have been already computed in \cref{eq:grad-standard-lagrangian},
\begin{equation*}
  \begin{aligned}
     \Grad L(q,w,t) &= \euler^{\gamma_t}\left(\euler^{\alpha_t}\left(\frac{e_q(\euler^{-\alpha_t}w)}{q} - 1\right) -
w - \euler^{\alpha_t+\beta_t} \Grad f(q)\right) \\
    \Grad_{\text{e}} L(q,w,t) &= \euler^{\gamma_t}\left(\frac{e_q(\euler^{-\alpha_t}w)}{q} - 1\right)
  \end{aligned}
\end{equation*}

The Euler-Lagrange equation is
\begin{multline*}
  \Derivby t \left(
    \euler^{\gamma_t}\left(\frac{e_q(\euler^{-\alpha_t}\velocity q(t))}{q} -
      1\right)\right) = \\
  \euler^{\gamma_t}\left(\euler^{\alpha_t}\left(\frac{e_q(\euler^{-\alpha_t}\velocity
        q(t))}{q} - 1\right) -
    \velocity q(t) - \euler^{\alpha_t+\beta_t} \Grad f(q(t))\right) \ ,
\end{multline*}
or, canceling the factor $\euler^{\gamma(t)}$,
\begin{multline}\label{eq:EL-time-dependent}
 \dot \gamma_t \left(\frac{e_q(\euler^{-\alpha_t}\velocity q(t))}{q} -
      1\right) +  \Derivby t \left(\frac{e_q(\euler^{-\alpha_t}\velocity q(t))}{q} -
      1\right) = \\
 \euler^{\alpha_t}\left(\frac{e_q(\euler^{-\alpha_t}\velocity
        q(t))}{q(t)} - 1\right) -
    \velocity q(t) - \euler^{\alpha_t+\beta_t} \Grad f(q(t)) \ .
\end{multline}
\

Let us compute the left-hand side. If we write $\chi(t) =
e_{q(t)}\left(\euler^{-\alpha_t \velocity q(t)}\right)$, then
\begin{equation*}
  \Derivby t 
    \left(\frac{\chi(t)}{q(t)} - 1
    \right) = \frac1{q(t)} \derivby t \left(\chi(t) - q(t)\right) =
    \frac{\chi(t)}{q(t)} \velocity \chi(t) - \velocity q(t) \ ,
\end{equation*}
where
\begin{multline*}
\velocity \chi(t) = \derivby t \left(\euler^{-\alpha_t} \velocity q(t) -
  K_{q(t)}(\euler^{-\alpha_t} \velocity q(t)) + \log q(t)\right) = \\
- \dot \alpha_t \euler^{-\alpha_t} \velocity q(t)  +
  \euler^{-\alpha_t} \derivby t \velocity q(t) - \derivby t
  K_{q(t)}(\euler^{-\alpha_t} \velocity q(t)) + \velocity q(t) = \\
  \left(1 - \dot \alpha_t \euler^{-\alpha_t}\right) \velocity q(t) +
  \euler^{-\alpha_t} \acceleration q(t) + \euler^{-\alpha_t} \expectat
  {q(t)} {\derivby t \velocity q(t)} - \derivby t
  K_{q(t)}(\euler^{-\alpha_t} \velocity q(t)) \ . 
\end{multline*}
It follows that
\begin{equation*}
\Derivby t 
    \left(\frac{\chi(t)}{q(t)} - 1
    \right) = \mtransport {\chi(t)}{q(t)} \etransport {q(t)}{\chi(t)} \left(\left(1 - \dot \alpha_t \euler^{-\alpha_t}\right) \velocity q(t) +
  \euler^{-\alpha_t} \acceleration q(t)\right) - \velocity q(t) \ ,  
\end{equation*}
and, in turn,  the Euler-Lagrange \cref{eq:EL-time-dependent} becomes
\begin{multline*}
  \dot \gamma_t \left(\frac{\chi(t)}{q(t)} - 1 \right) + \mtransport {\chi(t)}{q(t)} \etransport {q(t)}{\chi(t)} \left(\left(1 - \dot \alpha_t \euler^{-\alpha_t}\right) \velocity q(t) +
  \euler^{-\alpha_t} \acceleration q(t)\right) - \cancel{\velocity q(t)} = \\
 \euler^{\alpha_t}\left(\frac{\chi(t)}{q(t)} - 1\right) -
    \cancel{\velocity q(t)} - \euler^{\alpha_t+\beta_t} \Grad f(q(t)) \ .
  \end{multline*}

The equation above can be rearranged to read
\begin{multline*}
  \euler^{-\alpha_t} \acceleration q(t) +  \left(1 - \dot \alpha_t
    \euler^{-\alpha_t}\right) \velocity q(t) = \\
  (\euler^{\alpha_t}-\dot \gamma_t) \left(- \frac{q(t)}{\chi(t)} +
    \expectat {q(t)} {\frac{q(t)}{\chi(t)}}\right) - \\
     \euler^{\alpha_t+\beta_t} \left(\frac{q(t)}{\chi(t)}\Grad f(q(t)) -
       \expectat {q(t)} {\frac{q(t)}{\chi(t)}\Grad f(q(t))}\right) 
 \end{multline*}
\

We shall now see the Euler-Lagrange equation above as the solution of an optimization problem on the simplex, where the potential $f(q)$ represents the objective function to be minimized, with a \emph{proximity} condition induced by the KL divergence. The explicit time-dependence of the Lagrangian is the fundamental ingredient in order for the dynamical system to dissipate energy and relax to a minimum of the potential,  hence to a minimum of the objective function.

\begin{remark}
In the dynamics induced by the Bregman Lagrangian analogue to \cref{eq:breglagr} in $I \subset \mathbb{R} \times \mathbb{R}^{2d}$, \cite{WibisonoE7351} propose a set of \emph{ideal} scaling conditions $( \dot{\beta} \le e^{\alpha_t}, \dot{\gamma}= e^{\alpha_t})$, by which solutions of the Euler-Lagrange equations reduce to vanishing-step-size-limit trajectories of the accelerated gradient optimization schemes (see also \cite{JMLR:v17:15-084,NIPS2015_5843, Attouch2018}).
Applying the same ideal scaling conditions to the KL Euler-Lagrange equations on the statistical bundle leads to the simplified equation
\begin{equation}\label{bobo}
\acceleration q(t)  +(\euler^{\alpha_t}-\dot{\alpha}_t) \velocity q(t)= -
     \euler^{2\alpha_t+\beta_t} \left(\frac{q(t)}{\chi(t)}\Grad f(q(t)) -
       \expectat {q(t)} {\frac{q(t)}{\chi(t)}\Grad f(q(t))}\right)~.
\end{equation}
\     

Along with \cite{WibisonoE7351}, the following choice of parameters, indexed by $p>0$,
\begin{equation}
\label{parametr}
\alpha_t= \log p-\log t  \qquad\quad 
\beta_t=  p \log t+\log C \qquad\quad 
\gamma_t=p \log t, 
\end{equation}
where $C>0$ is a constant, can be associated to the continuous-time limit of Nesterov's accelerated mirror descent (when $p=2$) and that of Nesterov's accelerated cubic-regularized Newton's method (when $p=3$) on the simplex.  Fig.~\ref{simp3} shows the solution of the system of ODEs for \cref{bobo}. The explicit expression of the ODEs system is given in \cref{jordano} in appendix.


 
\end{remark}

\begin{figure}
\begin{center}
\includegraphics[width=12cm]{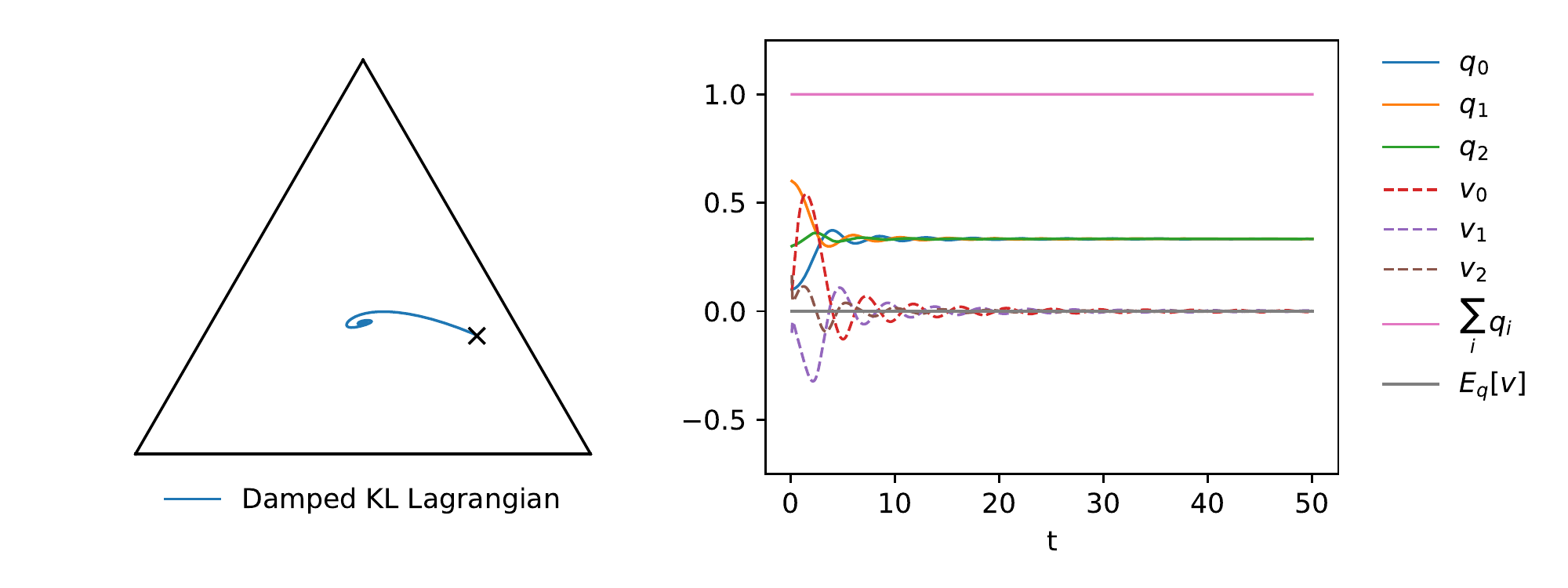}
\caption{Solution of the Euler-Lagrange equation for the damped KL Lagrangian in \cref{eq:breglagr}, under the ideal scaling condition, with the specific choice of parameterization given in \eqref{parametr} (with $p=2, C=0.5$). On the right panel, we see how the components of the probability oscillate toward the optimal value which minimize the potential. Accordingly, the velocities oscillate toward their vanishing expected value. The explicit expression of \eqref{bobo} as a system of ODEs is given in \cref{jordano} of the appendix.
}\label{simp3}
\end{center}
\end{figure}

\subsection{Damped KL Hamiltonian}
\label{example:hamiltonian}

The momentum corresponding to the \emph{damped} KL Lagrangian is given by
\begin{equation}\label{mom}
 {\eta} =  \euler^{\gamma_t}\, \left(\frac{e_q(\euler^{-\alpha_t}\,w)}{q}-1\right) \in I \times  \mixfiberat q \mu  \ ,
\end{equation}
corresponding to a time-dependent damping of the mechanic momentum associated to the free KL Lagrangian 
\begin{equation*}\label{tscaledKL}
 L(q,w,t)=\euler^{\alpha_{t}}\,D(q,e_q(\euler^{-\alpha_t}\,w)) \ .
\end{equation*}

Now, for
$(1+\euler^{-\gamma_t}\, {\eta})q = e_q(\euler^{-\alpha_t}\,w)$, we can use the exponential chart at $q$ to easily invert the Legendre transform and solve for the velocity $w$.  We have
\begin{multline*}
w({\eta})=(\Grad L_q)^{-1}( {\eta})= \euler^{\alpha_t}\, s_q((1+\euler^{-\gamma_t}\,{\eta})\, q) = \\ \euler^{\alpha_t}\,\Big(\log(1+\euler^{-\gamma_t}\,{\eta})
- \expectat q {\log(1+\euler^{-\gamma_t}\,{\eta})} \Big) \ .
\end{multline*}

Thereby, we can write the KL Hamiltonian as
\begin{multline*}
    H(q,{\eta},t) =  \euler^{\alpha_t}\, \scalarat q {{\eta}}{\Big(\log(1+\euler^{-\gamma_t}\,{\eta})
- \expectat q {\log(1+\euler^{-\gamma_t}\,{\eta})} \Big)}+\\
- \euler^{\alpha_t+\gamma_t}\,K_q\Big(\log(1+\euler^{-\gamma_t}\,{\eta})
- \expectat q {\log(1+\euler^{-\gamma_t}\,{\eta})} \Big)+ \euler^{\alpha_t+\beta_t+\gamma_t} f(q)  \\
=  \euler^{\alpha_t+\gamma_t}\, \left( \expectat q {\euler^{-\gamma_t}\,\eta \, \Big(\log(1+\euler^{-\gamma_t}\,{\eta})
- \expectat q {\log(1+\euler^{-\gamma_t}\,{\eta})} \Big)}+ \right.\\
\left.- K_q\Big(\log(1+\euler^{-\gamma_t}\,{\eta})
- \expectat q {\log(1+\euler^{-\gamma_t}\,{\eta})} \Big)\right)
+ \euler^{\alpha_t+\beta_t+\gamma_t} f(q)\\
= \euler^{\alpha_t+\gamma_t}\, \Big(\expectat q {(1+\euler^{-\gamma_t}\,{\eta}) \, \log(1+\euler^{-\gamma_t}\,{\eta})} + \euler^{\beta_t} f(q)\Big) =\\ \euler^{\alpha_t+\gamma_t}\, \Big(D(e_q(\euler^{-\alpha_t}w),q) + \euler^{\beta_t} f(q)\Big) \ ,
\end{multline*}

as presented in \eqref{eq:dampham}.

\begin{remark}
In classical mechanics, we normally identify the Hamiltonian with the total energy of the system, given by the sum of the kinetic and potential energy. In facts, the kinetic energy is the conjugate to the Lagrangian kinetic energy. This is apparent in the generalised case of the KL, where the symmetry of the standard quadratic form is generalised to a conjugacy relation with respect to the dual pairing. In the finite dimensional case, where the statistical bundle
coincides with the dual statistical bundle, the mechanic interpretation is then preserved.
\end{remark}

\begin{figure}
\begin{center}
\includegraphics[width=12cm]{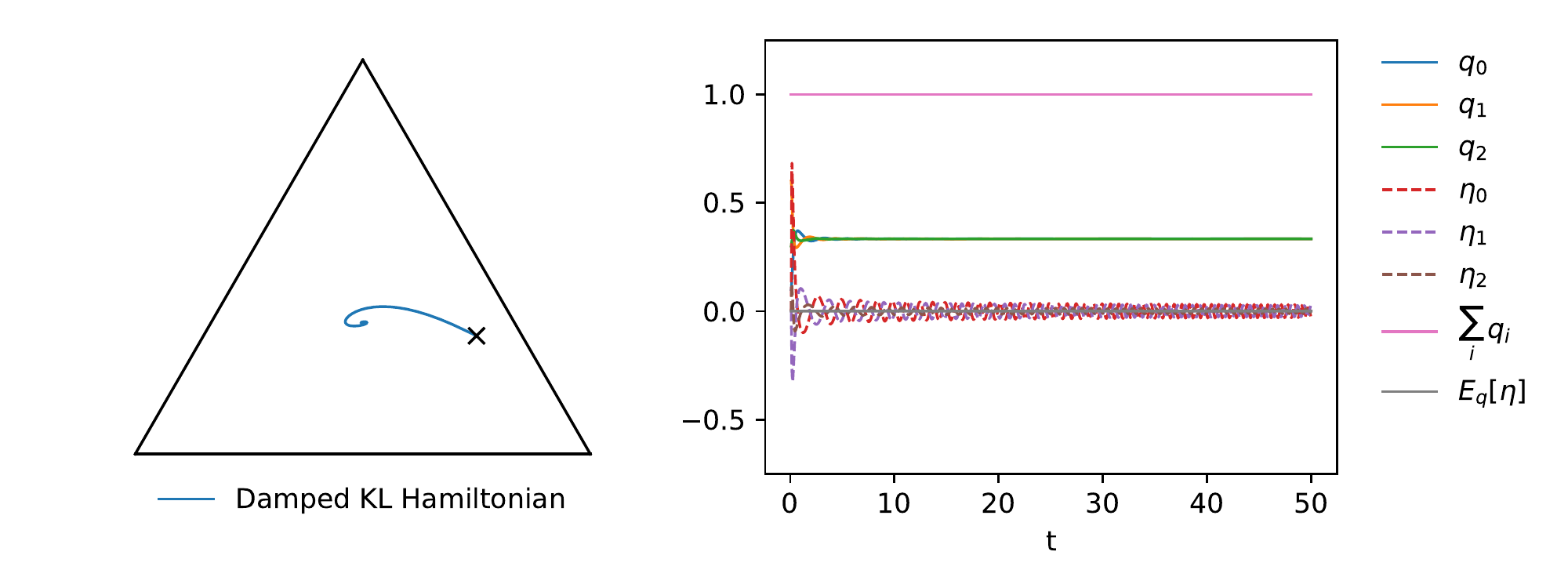}
\caption{Solution of the damped KL Hamiltonian system ODEs associated to \cref{eq:Hamiltondamp}. We use same ideal scaling condition, parameterization ($p=2, C=0.5$), and initial conditions of the KL Lagrangian system in~\cref{simp3}. On the right panel, we see that the damping of the probability densities is faster than the one shown by the KL Lagrangian system. The momenta oscillate toward their vanishing expected value, but they seem to tend to a finite amplitude. The explicit expression of the system of ODEs is given in \cref{eq:Hamiltondampeta-sm}.}\label{simp4}
\end{center}
\end{figure}

The {Hamilton equations} read
 \begin{equation}
  \label{eq:Hamiltondamp}
  \left \{
  \begin{aligned}
  \Derivby t {\eta} &=  \euler^{\alpha_t}\, {\eta} - \euler^{\alpha_t+\gamma_t}\left(\log(1+\euler^{-\gamma_t}\,{\eta}) - \expectat {q(t)}
    {\log(1+\euler^{-\gamma_t}\,{\eta})}\right) + \\ 
    & \qquad \qquad \qquad \qquad \qquad \qquad \qquad \qquad \qquad  - \euler^{\gamma_t+\alpha_t+ \beta_t} \Grad f(q) \\
     \velocity q(t) & = \euler^{\alpha_t}\,\left(\log(1+\euler^{-\gamma_t}\,{\eta}) - \expectat {q(t)}
  {\log(1+\euler^{-\gamma_t}\,{\eta})}\right) \  .
  \end{aligned}
  \right.
\end{equation}

As for \cref{eq:H-1}, we can express the system in the global space in which the dual
bundle is embedded by using 
$\Derivby t \eta(t) = \frac{\dot q(t)}{q(t)} \eta(t) + \dot \eta(t)$,
$\velocity q(t) = \frac{\dot q(t)} {q(t)}$ (see cf.~\cref{eq:Hamiltondampeta-sm}), \cref{jordano} in the appendix).

\begin{remark}
It is interesting to note that the momentum derived from the parametric Lagrangian in \eqref{eq:breglagr} (a generalization of the Kanai--Caldirola Lagrangian (see for example \cite{Bateman1931, Herrera1986}) is nothing but the mechanic conjugate momentum to $q(t)$, defined in Proposition \ref{ex:mech-cumulant-lagrangian} for the undamped system, multiplied by a scaling factor $\euler^{\gamma_t}$. In fact, despite giving the correct equation of motion for the damped harmonic oscillator, a Lagrangian of the type $L=\euler^{\gamma_t}\left(T-V \right)$ has been shown to rather describe a harmonic oscillator with variable mass \cite{doi:10.1063/1.524148}. This is an interesting perspective to be explored for understanding the role of inertia and acceleration in the momentum approaches for optimization.
\end{remark}

\section{Discussion}\label{sec:discussion}

This paper has the character of a first full analysis of a new formalism to be of interest in the study of the evolution of probability densities on a finite sample space.
We showed that a fully
non-parametric presentation of the Lagrangian and Hamiltonian dynamics
can be realized on the statistical bundle. Our version of the mechanical formalism applies to a set up that is different
from the standard one. Namely, it acts on that specific version of the
tangent bundle of the open probability simplex that has the most natural interpretation in terms of statistical quantities.

All the underlying mechanical concepts, such as velocity, parallel transport, accelerations, second-order equations, oscillation, damping, receive a specific statistical interpretation and are, in some cases, related with non-mechanical features, such as a divergence. Several simple examples illustrate the numerical implementation and graphical illustration of the results, which are both important in statistical and machine learning applications, as well as for a deeper understanding of accelerated methods 
and  the construction of new geometric discretization schemes for optimization algorithms (see, for example, \cite{zhang2018riemannian, alimisis2019continuoustime}).

The use of formal mechanical concepts in statistical modelling has been unusual in the applied literature, and we believe this paper could prompt for a change of perspective. On the other side, the non-parametric treatment clearly shows the relation with the modelling used in Statistical Physics that is, Boltzmann-Gibbs theory. 
The case of a generic exponential family is obtained by considering the vector space generated by the constant 1 and a finite set of
sufficient statistics, $B = \spanof{1,u_1,\dots,u_n}$. In this case
the fibers are $B_p=\setof{u \in B}{\expectat p u = 0}$. The
   exponential transport acts properly, while the mixture transport is
   computed as
   \begin{equation*}
     \scalarat q \eta {\etransport p q w} = \scalarat p {\Pi_p
       \mtransport q p \eta} w \ ,
   \end{equation*}
where $\Pi_p$ is the $L^2(p)$ orthogonal projection onto $B_p$. 

As a further direction for future research, the dynamical system induced by the Hamilton equations on the dual bundle suggests the study of the measures on the statistical bundle and their evolution. The consideration of such measures provides an interesting extension of the Bayes paradigm in that there is a probability measure on the simplex and a transition to a measure on the fiber. 
Future work should consider the extension of the statistical bundle mechanics to the continuous state space, which requires the definition of a proper functional set-up. One option would be to model the fibers of the statistical bundle with an exponential Orlicz space. In such a case, the fibers of the dual statistical bundle should be modeled as the pre-dual space. Many other options are available, in particular, Orlicz-Sobolev Banach spaces, or Fr\`echet spaces of infinitely differentiable densities.
An extension to the continuous state space and to arbitrary exponential families would allow a broad application of the information geometric formalism for accelerated methods to the optimization over statistical models, in particular in large dimensions. Three examples of such applications are the optimization of functions defined over the cone of the positive definite matrices, the minimization of a loss function for the training of neural networks, and the stochastic relaxation of functions defined over the sample space.
 
\section*{Acknowledgements}
G.C.~and L.M.~are supported by the DeepRiemann project, co-funded by the European Regional Development Fund and the Romanian Government through the Competitiveness Operational Programme 2014-2020, Action 1.1.4, project ID P\_37\_714, contract no. 136/27.09.2016, SMIS code 103321. G.P.~is supported by de Castro Statistics, Collegio Carlo Alberto, Turin, Italy. He is a member of GNAMPA-INDAM.

\newpage

\appendix

\section{Covering}
\label{sec:great-circles}

The mapping \eqref{eq:spheretosimplex} of the positive part of the sphere to the open simplex is actually a restriction of a smooth mapping of the full sphere to $\reals^N$ whose image is the closed simplex. The unrestricted mapping provides a topological covering of the simplex. Hence, the geodesic dynamics of the sphere is naturally mapped to a dynamics of the simplex. The following example shows how this particular behaviour is expressed in our framework.

The expression in the chart centered at $p$ of the quadratic
Lagrangian $L(q,w) = \frac m2 \scalarat q w w$ follows from
\cref{eq:K3},
\begin{equation}\label{eq:free-particle-in-chart}
L_p(u,v) = \frac{m}2 \scalarat {e_p(u)} {\etransport p {e_p(u)}
v} {\etransport p {e_p(u)} v} =  \\ \frac{m}2 D^2 K_p(u)[v,v]  \ ,
  \end{equation}
where $q = e_p(u)$ and $w = \etransport p q v$.

The derivative of $L_p(u,v)$ in the direction $(h,k)$ is
\begin{multline}
  D L_p(u,v)[(h,k)] = \\ 
  \frac m2 D^3K_p(u) [v,v,h] + m D^2K_p(u)[v,k] = 
  \frac m2 \expectat q {w^2 \etransport p q h} + m \expectat q {w
    \etransport p q k} = \\ \frac m2 \scalarat q {w^2 - \expectat q {w^2}}
  {\etransport pq u} + m \scalarat q w {\etransport p q k}\ ,
\end{multline}
where we have used \cref{eq:K4}.

The two components of the natural gradient are
\begin{equation*}
  \Grad L(q,w) = \frac m2 (w^2 - \expectat q {w^2}) \ , \quad
  \Grad_e L(q,w) = m w \ .
\end{equation*}
With $w(t)=\velocity q(t)$, Euler-Lagrange equation is
\begin{equation*}
  \Derivby t \velocity q(t) = \frac 12 \left({\velocity q(t)}^2 - \expectat {q(t)}{{\velocity q(t)}^2}\right) \ .
\end{equation*}
Notice that the $\velocity q(t)^2 - \expectat {q(t)}{{\velocity
    q(t)}^2}$ belongs to the fiber $\mixfiberat {q(t)} \mu$ of the
dual bundle.

Let us express this equation as a system of second-order ODEs. By recalling that here $D/dt$ is the dual (mixture) covariant derivative, we find that
\begin{equation*}
  \frac{\ddot q(t)}{q(t)} =
  \frac12 \left(\frac{\dot q(t)}{q(t)}\right)^2 - \frac12 \expectat
  {q(t)} {\left(\frac{\dot q(t)}{q(t)}\right)^2} \ .   
\end{equation*}

If $\Omega = \set{1,\dots,N}$, this is a system of $N$ 2nd order ODE,
\begin{equation*}
  \ddot q(j;t) = \frac {\dot q(j;t)^2} {2q(j;t)} - \frac{q(j;t)}{2N}
  \sum_{i=1}^N \frac{\dot q(i;t)^2}{q(i;t)} \ , \quad j = 1,\dots, N\ . 
\end{equation*}

A second option is to take as an indeterminate $\ell(t) = \log
q(t)$. It follows that
\begin{equation*}
  \dot \ell(t) = \velocity q(t) \quad \text{and} \quad \Derivby t
  \velocity q(t) = \ddot \ell(t) + \dot \ell(t)^2 \ ,
\end{equation*}
and the equation becomes
\begin{equation*}
  \ddot\ell(j;t) = - \frac12 \dot \ell(j;t)^2 - \frac1{2N}
  \sum_{i=1}^N \euler^{\ell(i;t)} \dot\ell(i;t)^2 \ . 
\end{equation*}

There is a closed form solution. In fact, the solution is the image of the Riemannian exponential on the sphere of radius 1 by a proper transformation.

The mapping
\begin{equation*}
\expbundleat \mu (q,w) \mapsto ((q/N)^{1/2},(q/N)^{1/2} w) = (\alpha,\beta) \in TS_1 \ ,
\end{equation*}
{where $TS_1$ is the tangent bundle of the unit sphere $S_1$}, is well defined, is 1-to-1 onto the positive quadrant, and preserves the respective metrics. In fact,
\begin{gather*}
  TS_{1+} \ni (\alpha,\beta) \mapsto (N\alpha^2,\alpha^{-1}\beta) \in
  \expbundleat \mu \ , \\
  \sum_x \alpha^2(x) = \sum_x q(x) \frac 1N = 1 \ , \\
  \sum_x \alpha(x) \beta(x)   = \sum_x w(x) q(x) \frac 1N = \expectat
  q w = 0 \ , \\
  \sum_x \beta_1(x) \beta_2(x) = \sum_x w_1(x) w_2(x) q(x) \frac1N =
  \scalarat q {w_1}{w_2} \ 
\end{gather*}
Notice that this is \emph{not} the Amari's embedding that maps
$T\Delta^\circ$ onto $TS_1$.

The exponential mapping on the Riemannian manifold $S_1$ is given by the geodesic defined for each $(\alpha,\beta) \in TS_1$ by
\begin{equation*}
 \alpha(t) = \cos(\normof \beta  t) \alpha +  \normof \beta ^ {-1} \sin(\normof
 \beta  t) \beta \ .
\end{equation*}
We have  $\alpha(0) = \alpha$,
\begin{gather*}
  \normof {\alpha(t)} ^ 2 = \cos^2(\normof \beta ^ 2 t) +  \normof \beta ^
  {-2} \normof \beta ^ 2  \sin^2(\normof \beta ^ 2 t) = 1 \ ,  \\
  \dot \alpha(t) = - \normof \beta \sin(\normof \beta t) \alpha +
  \cos(\normof \beta t) \beta \ , \quad \dot \alpha(0) = \beta \ , \\
  \dot \alpha (t) \cdot \alpha(t) = - \normof \beta \cos(\normof \beta
  t) \sin(\normof \beta t) \alpha \cdot \alpha + \normof \beta ^ {-1}
  \sin(\normof \beta t) \cos(\normof \beta t) \beta \cdot \beta  = 0
  \ , \\
  \normof {\dot \alpha(t)} ^2 = \normof \beta ^ 2 \sin^2(\normof \beta
    t) \normof \alpha^2 + \cos^2(\normof \beta t) \normof \beta ^ 2 =
    \normof \beta ^ 2 \ . 
  \end{gather*}
  We leave out the check it is a geodesic.

Given $(q,w) \in \expbundleat \mu$, let us apply the isometric
transformation, and define
\begin{multline*}
  q(t) = \\ N \left(\cos(\normof{(q/N)^{1/2} w}  t) (q/N)^{1/2} +
    \normof {(q/N)^{1/2} w} ^ {-1} \sin(\normof {(q/N)^{1/2} w}  t)
    (q/N)^{1/2} w \right)^2 = \\ q\left(\cos( \sigma(w)  t) +
     \sigma(w) ^ {-1} \sin( \sigma(w)  t)
   w \right)^2 \ ,
\end{multline*}
with
\begin{equation*}
  \sigma(w) = \normof {(q/N)^{1/2} w} = \sqrt{\sum _x w(x) q(x) \frac1N}  =
    \sqrt{\expectat q {w^2}} \ .
\end{equation*}

We assume that $t \in I$ and $q(t) > 0$. Let us compute the
velocity and the acceleration  of $t \mapsto q(t)$. It holds $q(0) = q$ and
\begin{gather*}
\velocity q(t) = \frac{2 (w \cos (\sigma t)-\sigma \sin (\sigma
  t))}{\sigma^{-1} w \sin (\sigma t) +\cos (\sigma t)} \\
\velocity q(t) ^2 = \frac{4 (w \cos (\sigma  t)-\sigma 
   \sin (\sigma  t))^2}{\left(\sigma^{-1} w \sin (\sigma t) +\cos (\sigma t)\right)^2}\\
\ddot q(t) / q(t) = -\frac{2 \sigma ^2 \left(\left(\sigma
   ^2-w^2\right) \cos (2 \sigma  t)+2
   \sigma  w \sin (2 \sigma 
   t)\right)}{(\sigma  \cos (\sigma 
 t)+w \sin (\sigma  t))^2} \\
\frac{\ddot q(t)}{q(t)} - \frac12 \velocity q(t)^2 = - 2 \sigma^{-2} \\
   \velocity q(t)^2 q(t) = 4 q (w \cos (\sigma  t)-\sigma  \sin
   (\sigma  t))^2 
\end{gather*}

It follows that the Riemannian acceleration on the Hilbert bundle
is null.
%
%

\section{Entropy Flow}\label{entropy}
We compute the natural gradient of the entropy  $\entropyof q = - \expectat q {\log q}$ by using the Hessian formalism. Notice that, by definition, $\log q + \entropyof q \in \expfiberat q \mu$.

If
\begin{equation}
  t \mapsto q(t) = \euler^{v(t)-K_p(v(t))}\cdot p \ , \quad v(t) = s_p(q(t)) \ ,
\end{equation}
is a smooth curve in $\maxexpat \mu$ expressed in the chart centered at $p$, then we can write
  \begin{multline} \label{eq:HP}
    \entropyof {q(t)} = - \expectat {q(t)} {v(t) - K_p(v(t)) + \log p} = \\ K_p(v(t)) - \expectat {q(t)}{v(t) + \log p + \entropyof p} + \entropyof p = \\
K_p(v(t)) - D K_p(v(t))[v(t) + \log p + \entropyof p] + \entropyof p    \ ,
  \end{multline}
where the argument $v(t) + \log p + \entropyof p$ of the expectation belongs to the fiber $\expfiberat P \mu$ and we have expressed the expected value as a derivative by using Eq~\eqref{eq:K1}.

By using Eq~\eqref{eq:K1} and Eq~\eqref{eq:K3}, we see that the derivative of the entropy along the given curve is
\begin{multline*}
  \derivby t \entropyof {q(t)} = \derivby t K_p(v(t)) - \derivby t  D K_p(v(t))[v(t) + \log p + \entropyof p] = \\
  D K_p(v(t))[\dot v(t)] - D^2K_p(v(t))[v(t) + \log p + \entropyof p,\dot v(t)] - D K_p(v(t))[\dot v(t)] = \\ - \expectat {q(t)}{\etransport p {q(t)} (v(t) + \log p + \entropyof p)\etransport p {q(t)} \dot v(t)} \ .
\end{multline*}

We use then the identities
\begin{gather}
  v(t) + \log p + \entropyof p = \log q(t) + K_p(v(t)) + \entropyof p \ , \\ \etransport p {q(t)} \left(\log q(t) + K_p(v(t)) + \entropyof p\right) = \log q(t) + \entropyof {q(t)} \ , \\
  \etransport p {q(t)} \dot v(t) = \velocity q(t) \ ,
\end{gather}
to obtain
\begin{equation}
    \derivby t \entropyof {q(t)} = - \scalarat {q(t)} {\log q(t) + \entropyof {q(t)}}{\velocity q(t)} \  .
  \end{equation}
Hence, we can identify the gradient of the entropy in the statistical bundle as
  \begin{equation}\label{eq:gradH}
\Grad \entropyof q = -(\log q + \entropyof q) \ .    
  \end{equation}

  The integral curves of the gradient flow equation
\begin{equation*}
\velocity q(t) = \Grad \entropyof {q(t)}  
\end{equation*}
are exponential families of the form $q(t) \propto
q(0)^{\euler^{-t}}$. In fact, if we write the equation in $\reals^N$,
we get the quasi-linear ODE
\begin{equation*}
 \derivby t \log q(t) = -\log q(t) - \entropyof {q(t)} \ ,  
\end{equation*}
and, in turn,
\begin{equation*}
  \log q(t) = \euler^{-t} \log q(0) - \euler^{-t} \int_0^t \euler^s
  \entropyof {q(t)} \ ds \ .
\end{equation*}
The behaviour as $t \to \pm \infty$ and other properties follow easily~\cite{pistone:2013Entropy}.

Given a section $q \mapsto F(q) \in \expfiberat q \mu$, the variation of the entropy along the integral curves, $\velocity q(t) = F(q(t))$, is
\begin{equation*}
  \derivby t \entropyof{q(t)} =  \scalarat {q(t)} {\Grad \entropyof
    {q(t)}} {F(q(t))} = - \scalarat {q(t)} {\log q(t) + \Grad \entropyof
    {q(t)}} {F(q(t))} \ .
  \end{equation*}
  For example, the condition for entropy production becomes
  \begin{equation*}
    \scalarat {q} {\log q + \entropyof {q(t)}} {F(q)} = \expectat q {\log q F(q)} < 0 \ .
  \end{equation*}

\section{Simple Examples} \label{sm:computations}

Here, we have collected simple examples of Lagrangian and Hamiltonian
dynamics of the simplex.

\begin{example}[Quadratic Lagrangian]\label{ex:quadratic-lagrangian}
  If $L(q,w) = \frac12 \scalarat q w w$, then
  \begin{equation}\label{eq:quadratic-lagrangian-p}
    L\left(e_p(u),\etransport p {e_p(u)} v\right) = \frac12 \expectat {e_p(u)}
  {\left(\etransport p {e_p(u)} v\right)^2} = \frac12 \expectat p {\euler^{u -
      K_p(u)}\left(\etransport p {e_p(u)} v\right)^2} \ ,
  \end{equation}
  with derivative with respect to $u$ in the direction $h$ given by
  \begin{multline*}
 \frac12 \expectat p  {\euler^{u -
      K_p(u)} \etransport p {e_p(u)} h \left(\etransport p {e_p(u)} v\right)^2} +\\ \expectat p {\euler^{u -
      K_p(u)}\left(\etransport p {e_p(u)} v\right)(- \covat {e_p(u)} v
  h)} = \\ \frac12 \expectat q {w^2 \etransport p {e_p(u)} h} =
\frac12 \scalarat q {w^2 - \expectat q {w^2}}{\etransport p {e_p(u)} h}
  \ ,    
  \end{multline*}
which, in turn, identifies the natural gradient as $\Grad \frac12
\scalarat q w w = \frac12 (w^2 - \expectat q {w^2})$.
\end{example}

\begin{example}[Cumulant functional]\label{ex:cumulant-lagrangian}
If  $L(q,w) = K_q(w)$, then  
  \begin{multline}\label{eq:cumulant-lagrangian-p}
    L\left(e_p(u),\etransport p {e_p(u)} v\right) =
    K_{e_p(u)}\left(\etransport p {e_p(u)} v\right) = 
    \log \expectat {e_p(u)}  {\euler^{v - \expectat {e_p(u)} v}} = \\
    \log\expectat p {\euler^{u - K_p(u) + v - \expectat {e_p(u)} v}} =
    \log\expectat p {\euler^{u + v - K_p(u)- D K_p(u)[v] }} = \\
    K_p(u+v) - K_p(u) - D K_p(u)[v] \ .
  \end{multline}
  Notice that last member of the equalities is the Bregman divergence
  of the convex function $K_p$. The derivative with respect to $u$ in
  the direction $h$ is
  \begin{multline*}
    D K_p(u+v)[h] - D K_p(u)[h] - D^2K_p(u)[v,h] = \\ \expectat
    {e_p(u+v)} h - \expectat {e_p(u)} h - \expectat {e_p(u)}
    {\left(\etransport p {e_p(u)} v\right) \left(\etransport p {e_p(u)} h\right) } =
    \\ \expectat {e_p(u)} {\frac {e_p(u+v)}{e_p(u)} h} - \expectat {e_p(u)} h -
    \expectat q {w \left(\etransport p {e_p(u)} h\right)} = \\
 \expectat {e_p(u)} {\frac {e_p(u+v)}{e_p(u)} \left(\etransport p
     {e_p(u)} h\right)} - \scalarat  q {w} {\etransport p {e_p(u)} h}   \ .
  \end{multline*}
  The first term is
  \begin{multline*}
    \expectat {e_p(u)} {\euler^{v - (K_p(u+v)-K_p(u))}
      \left(\etransport p {e_p(u)} h\right)} = \\ \expectat {e_p(u)}
    {\euler^{v - (K_{e_p(u)}(\etransport p {e_p(u)} v) + D K_p(u)[v])}
      \left(\etransport p
        {e_p(u)} h\right)} = \\
    \expectat {e_p(u)} {\euler^{\etransport p {e_p(u)} v -
        K_{e_p(u)}(\etransport p {e_p(u)} v)} \left(\etransport p
        {e_p(u)} h\right)} = \\
    \expectat q {\euler^{w - K_q(w)} \left(\etransport p {e_p(u)}
        h\right)} = \scalarat q {\frac{e_q(w)}{q} - 1}{\etransport p
        {e_p(u)} h}  \ .
  \end{multline*}
In conclusion, $\Grad K_q(w) = \left(\frac{e_q(w)}{q} - 1\right) -
w$. The fiber gradient is easily seen to be 
$$\Grad_{\text{e}} K_q(w) = \frac{e_q(w)}{q} - 1 \ .$$ 

In the notations of Example \ref{ex:chi-trick},
$\chi(t) = e_{q(t)}(\velocity q(t))$, we have, for example,
\begin{multline*}
  \derivby t K_{q(t)}(\velocity q(t)) = \scalarat {q(t)} {\frac{\chi(t)}{q(t)}-1 - \velocity q(t)}{\velocity q(t)} + \scalarat {q(t)} {\frac{\chi(t)}{q(t)}-1}{\acceleration q(t)} = \\
  \expectat {\chi(t)} {\acceleration q(t) + \velocity q(t)} - \expectat {q(t)}{{\velocity q(t)}^2} \ .
\end{multline*}

This example shall be of interest for us because it is connected with
the KL divergence, $K_q(w) = D(q,{e_q(w)})$. 
\end{example}

\begin{example}[Conjugate cumulant functional]\label{ex:cumulant-hamiltonian}
The Hamiltonian
\begin{equation*}
 \mixbundleat \mu \colon (q,\eta) \mapsto H(q,\eta) = \expectat q {(1+\eta) \log (1 + \eta)} \ , \quad \eta > -1 \ ,   
\end{equation*}
is the Legendre transform of the cumulant function $K_q$,
\begin{equation*}
  H(q,\eta) = \scalarat q \eta {(\Grad K_q)^{-1}(\eta)} -
  K_q\left((\Grad K_q)^{-1}(\eta)\right) \ .
\end{equation*}
In particular, the fiber gradient of $H_q$ is $\Grad_{\text{m}} H(q,\eta) = \log(1+\eta) - \expectat q {\log(1+\eta)}$ which is the inverse of the fiber gradient of $K_q$. Notice that $r = (1+\eta)q$ is a density, and $D(r,q) = H(q,\eta)$. 

Let us compute the natural gradient. The expression of the Hamiltonian in the chart at $p$ is 
\begin{multline*}
 H_p(u,\zeta) = \expectat {e_p(u)} {\left(1 + \frac p {e_p(u)} \zeta \right) \log \left(1 + \frac p {e_p(u)} \zeta \right)}  = \\ \expectat p {\left(\frac{e_p(u)} p  + \zeta \right) \log \left(1 + \frac p {e_p(u)} \zeta \right)} \ .  
\end{multline*}
As, for $h \in \expfiberat p \mu$,
\begin{equation*}
  D \left(\frac{e_p(u)} p  + \zeta \right) [h]= \frac{e_p(u)}p
  \etransport p {e_p(u)} h \ \ \text{and} \ \  D_1 \left(1 + \frac
    p {e_p(u)} \zeta \right) [h]= - \frac p{e_p(u)} \zeta \etransport p
  {e_p(u)}     h \ ,
\end{equation*}
the derivative of $H_p$ with respect to $u$ in the direction $h$ is given by
\begin{multline*}
  D H_p(u,\zeta) [h]= \expectat p {\left(\frac{e_p(u)}p \etransport p
      {e_p(u)} h\right) \log \left(1 + \frac p {e_p(u)}
      \zeta \right)} + \\-
  \expectat p {\left(\frac{e_p(u)} p + \zeta \right) \left(1 + \frac p
      {e_p(u)} \zeta \right)^{-1}\frac p{e_p(u)} \zeta \etransport p
    {e_p(u)}     h} = \\
  \expectat q {\log(1+\eta)  \etransport p {e_p(u)} h} - \expectat q
  {\zeta \etransport p {e_p(u)} h}\ ,
\end{multline*}
hence $\Grad H(q,\eta) = \log(1+\eta) - \expectat q {\log(1+\eta)} - \eta$.
\end{example}

\begin{example}[Mechanics of the quadratic Lagrangian]\label{ex:mech-quadratic-lagrangian}\label{eg5}
  If $L(q,w) = \frac12 \scalarat q w w$, then the Legendre transform
  is $H(q,\eta) = \frac 12 \scalarat q \eta \eta$. The gradients are
  \begin{align*}
    \Grad H(q,\eta) &= -\frac12 \left(\eta^2 - \expectat q {\eta^2}\right)\\
    \Grad_{\text{m}} H(q,\eta) &= \eta\\
    \Grad L(q,w) &= \frac12 (w^2 - \expectat q {w^2}) \\
    \Grad_{\text{e}} L(q,w) &= w
  \end{align*}

  For $\velocity q=w \in \mixbundleat \mu$, the Euler-Lagrange equation is
  \begin{equation*}
    \Derivby t \velocity q(t) = \frac12 \left(\velocity q(t)^2 -
      \expectat {q(t)} {\velocity q(t)^2}\right) \ ,
  \end{equation*}
where the covariant derivative is computed in $\mixbundleat \mu$, that
is, $\Derivby t \velocity q(t) = \ddot q(t) / q(t)$. In terms of the
exponential acceleration $\acceleration q(t) = \ddot q(t) / q(t) - \left(\velocity q(t)^2 -
      \expectat {q(t)} {\velocity q(t)^2}\right)$, the Euler-Lagrange
      equation reads
      \begin{equation*}
        \acceleration q(t) = - \frac12 \left((\velocity q(t))^2 -
      \expectat {q(t)} {(\velocity q(t))^2}\right) \ ,
      \end{equation*}
while in terms of the Riemannian acceleration in \cref{eq:0acc}, it
holds $\acc q(t) = 0$.

The Hamilton equations are
\begin{equation*}
  \left\{
    \begin{aligned}
      \Derivby t \eta(t) &= \frac12 \left(\eta^2 - \expectat q {\eta^2}\right) \\
      \velocity q(t) &= \eta(t)
    \end{aligned}
\right. \ ,
\end{equation*}
with the covariant derivative again computed in $\mixbundleat \mu$.

The conserved energy is
\begin{equation*}
  H(q(t),\eta(t)) = \frac12 \scalarat {q(t)} {\velocity q(t)}{\velocity q(t)} = \frac12
  \expectat 1 {\frac {{\dot q(t)}^2}{q(t)}} \ .
\end{equation*}

In fact,
this variational problem has a closed-form solution which is the image of a
geodesic (great circle) on the sphere through an isometric covering from the tangent
bundle of the sphere to the statistical bundle, see
\cref{sec:great-circles}. It is interesting to note that this solution
is a periodic curve in the set of all densities, but consists of
different sections in $\maxexpat \mu$ because it is interrupted when it
touches tangentially the border of the simplex of probability densities. 
\end{example}



\begin{example}[Mechanics of the cumulant Lagrangian]\label{ex:mech-cumulant-lagrangian}
If $L(q,w) = K_q(w)$, then its Legendre transform is $H(q,\eta) =
\expectat q {(1+\eta) \log (1+\eta)}$. This is an expression of the
dual divergence: as $\eta =
\left(\frac{r}{q}-1\right)$ with $r = e_q(w)$, then $H(q,\eta)= D(r, q)$, namely the relative entropy dual to the cumulant. 

The gradients are

  \begin{align*}
    \Grad H(q,\eta) &= \log(1+\eta) - \expectat q {\log(1+\eta)} - \eta\\
    \Grad_{\text{m}} H(q,\eta) &= \log(1+\eta) - \expectat q {\log(1+\eta)}\\
    \Grad L(q,w) &= \Grad K_q(w) = \left(\frac{e_q(w)}{q} - 1\right) -
w \\
    \Grad_{\text{e}} L(q,w) &= \Grad_{\text{e}} K_q(w) = \left(\frac{e_q(w)}{q} - 1\right) \ .
  \end{align*}

The Euler-Lagrange equation is
\begin{equation}
  \Derivby t \left(\frac {\chi(t)}{q(t)}-1\right) = \left(\frac
    {\chi(t)}{q(t)}-1\right) - \velocity q(t) \ , \quad \chi(t) =
  e_{q(t)}(\velocity q(t)) \ .
\end{equation}

The Hamilton equations are
\begin{equation}
\left\{\begin{aligned}
\Derivby t \eta(t) &= - \log(1+\eta(t)) + \expectat {q(t)} {\log(1+\eta(t))} + \eta(t) \\
\velocity q(t) &=  \log(1+\eta) - \expectat q {\log(1+\eta)}
\end{aligned}
\right.    
\end{equation}

The conserved energy is
\begin{equation*}
  H(q(t),\eta(t)) = \expectat {q(t)} {\frac{\chi(t)}{q(t)} \log
    \frac{\chi(t)}{q(t)}} = D(\chi(t),q(t)) \quad \text{with} \quad
    \chi(t) = e_{q(t)}(\velocity q(t)) \ . 
\end{equation*}
\end{example}

\section{Covariant Time-Derivative of the KL Legendre Transform}\label{sec:chart-fiber-dev}
We report here the explicit calculation of the covariant time derivative 
\begin{equation*}
    \Derivby t \left( \euler^{\velocity q (t)-K_{q(t)}(\velocity q (t))}-1 \right) \ ,
\end{equation*}
derived by expressing $\Grad_1 K_q(w)$ in the p-chart, contracted with in the direction $\dot u \in \expfiberat p \mu$. 
The calculation gives an explicit example of the formalism adopted while dealing with a \emph{triple covariance}. We have
\begin{multline*}
   \derivby t \scalarat {q(t)} {\euler^{\velocity q(t) - K_q(\velocity q(t))}-1}{\etransport p {q(t)} \dot u(t)} = \\ 
   \derivby t \left(D K_p(u(t)+v(t))[\dot u(t)] - D K_p(u(t))[\dot u(t)]\right) = \\ D^2K_p(u(t)+v(t))\left[\dot u(t), \dot u(t)\right] +  D^2K_p(u(t)+v(t))\left[\dot u(t),\dot v(t)\right]- \\ D^2K_p(u(t))\left[\dot u(t), \dot u(t)\right] = \\      
   \expectat {e_p(u(t)+v(t))} { \etransport p {e_p(u(t)+v(t))}\, \dot u(t) \, \etransport p {e_p(u(t)+v(t))}\, \left(\dot u(t)+\dot v(t)\right) } - \\                   \scalarat {q(t)} {\etransport p {q(t)}\, \dot u(t)}{\etransport p {q(t)}\,\dot u(t)} = \\                         
   \expectat {e_p(u(t))} {\etransport {e_p(u(t)+v(t))}{e_p(u(t))} \etransport p {e_p(u(t)+v(t))} \dot u(t) \mtransport {e_p(u(t)+\
v(t))}{e_p(u(t))} \etransport p {e_p(u(t)+ v(t))} \, \left(\dot u(t) + \dot v(t)\right)}  - \\ 
\scalarat {q(t)} {\etransport p \
{q(t)}\, \dot u(t)}{\etransport p {q(t)}\,\dot u(t)} \\                         \expectat {e_p(u(t))} {\etransport p {e_p(u(t))} \, \dot u(t) \, \mtransport {e_p(u(t)+v(t))}{e_p(u(t))} \etransport p {e_p(u(t)+v(t)\
)} \left(\dot u(t)+\dot v(t)\right) } - \\                                  \scalarat {q(t)} {\etransport p {q(t)}\, \dot u(t)}{\etransport p {q(t)}\,\dot u(t)} = \\                                        
\scalarat {q(t)} {\etransport p {e_p(u(t))}\, \dot u(t)}{\mtransport {e_p(u(t)+v(t))}{e_p(u(t))} \etransport p {e_p(u(t)+v(t))}
 \left(\dot u(t)+\dot v(t)\right)}  - \\                               
 \scalarat {q(t)} {\etransport p {q(t)}\, \dot u(t)}{\etransport p {q(t)}\,\dot u(t)} = \\                                      
 \scalarat {q(t)} {\etransport p {e_p(u(t))}\, \dot u(t)}{\mtransport {e_p(u(t)+v(t))}{e_p(u(t))} \etransport {e_p(u(t))} {e_p(u(t)+v(t))}  \etransport p {e_p(u(t))} \, \Big(\dot u(t)+\dot v(t)\Big)} - \\   \scalarat {q(t)} {\etransport p {q(t)}\, \dot u(t)}{\etransport p {q(t)}\,\dot u(t)} = \\                                     
\scalarat {q(t)} {\etransport p {e_p(u(t))}\, \dot u(t)}{\mtransport {e_p(u(t)+v(t))}{e_p(u(t))} \etransport {e_p(u(t))} {e_p(u(t)+v(t))}  \left(\velocity q(t)+\,\acceleration q(t)\right)} - \\
\scalarat {q(t)} {\etransport p {e_p(u(t))}\, \dot u(t)}{\velocity q(t)}  \ .                                  
   \end{multline*}
   We have 
   \begin{multline*}
      \mtransport {e_p(u(t)+v(t))}{e_p(u(t))} \etransport {e_p(u(t))} {e_p(u(t)+v(t))}  \, \Big(\velocity q(t)+\,\acceleration q(t)\Big) \\
      =  \mtransport {e_p(u(t)+v(t))}{e_p(u(t))}\, \Big(\velocity q(t)+\,\acceleration q(t)- \expectat {e_p(u(t)+v(t))}{\velocity q(t)+\,\acceleration q(t)}\Big)\\
      = \frac{e_p(u(t)+v(t))}{e_p(u(t))} \Big(\velocity q(t)+\,\acceleration q(t)\,-\, \expectat {e_p(u(t)+v(t))}{\velocity q(t)+\,\acceleration q(t)}\Big)
   \end{multline*}


Hence, eventually, we get
\begin{equation*}
    \Derivby t \left( \euler^{\velocity q (t)-K_{q(t)}(\velocity q (t))}-1 \right)= \frac{e_p(u(t)+v(t))}{e_p(u(t))} \Big(\velocity q(t)+\,\acceleration q(t)
    - \expectat {e_p(u(t)+v(t))}{\velocity q(t)+\,\acceleration q(t)}\Big) \ .
\end{equation*} 

\section{ODEs Systems for the Figures in the Article}\label{odes}

We provide here the explicit ODEs expressions for the Lagrangian dynamics on the simplex derived in the main text. There are many ways to rewrite the Euler-Lagrange equations on the simplex as a system of ordinary differential equations in $\reals^{2N}$. An immediate option is to introduce the variables $q$ and $v =\velocity q$, in which case the solution will stay in the Grassmannian manifold $\sum_x q(x) v(x) = 0$.

\subsection{Quadratic Lagrangian Systems - Figures 1,2}\label{figure12}
Consider first the Euler-Lagrange equation for the motion of a probability
density within a potential on the statistical manifold (cf. \cref{eq:dynamic-example} in the article). This is given by
\begin{equation}\label{q-potential}
  m {\Derivby t \velocity q(t)} = \frac m 2\left(\velocity q(t)^2 -
     \expectat{q(t)}{\velocity q(t)^2}\right) - \kappa (\log q(t) + \entropyof {q(t)} \ ,
 \end{equation} 
Let us express \cref{q-potential} as a system of ordinary differential equations for $q$ and $\velocity q$. We write $v(t) = \velocity q(t)$ and note that $v(t) = \derivby t \log q(t)$ implies
\begin{equation}\label{66}
\derivby t q(x;t) = q(x;t) v(x;t) \ , \quad x \in \Omega \ .
\end{equation}
In particular, \cref{66}, together with the assumption $\sum_x q(x;t) = N$ implies $\expectat {q(t)} {v(t)} = 0$. Conversely, if $\expectat {q(t)} {v(t)} = 0$, then $\sum_x q(x;t) = \sum_x q(x;0)$.

We have
\begin{equation}
\derivby t v(t) = \derivby t \frac {\dot q(t)}{q(t)} = \frac {\ddot q(t) q(t) - \dot q(t)^2}{q(t)^2} = \frac {\ddot q(t)}{q(t)} - {v(t)^2} \ ,
\end{equation}
where the first term on the rhs is the mixture acceleration of \cref{eq:macc}.
Thereby, via \cref{q-potential}, we have 
\begin{equation}
\derivby t v(t) = \frac {\ddot q(t)}{q(t)} - {v(t)^2} =  \Derivby t v(t)- v(t)^2\\
= -\frac 1 2 v(t)^2 - \expectat{q(t)}{\frac{1}{2} v(t)^2}  + \frac{\kappa}{m} \Grad \entropyof q  \ .
\end{equation}

If we write the vectors $A(q,v) = v^2/2+ \frac \kappa m \logof q $ and $B(q,v) = v^2/2 - \frac \kappa m \logof q$, the system of first order differential equations is
\begin{equation}\label{ds1}
  \left\{
    \begin{aligned}
    \derivby t q(x;t) &= q(x;t)v(x;t) \\
    \derivby t v(x;t) &= - A(q(x;t),v(x;t)) - \frac1N \sum_y q(y;t) B(q(y;t),v(y;t))
  \end{aligned}
  \right. \ , \quad x \in \Omega \ .
\end{equation}

The case of a free-particle Lagrangian system, whose solution is pictured in \cref{fig:free-motion} in the article, is obtained by setting $\kappa=0$. 

\subsection{Kullback-Leibler Lagrangian System - Figure 2}\label{KLode}
Consider now the family of parametrised Lagrangians given in \cref{eq:standard-lagrangian} in the article. We have
\begin{equation}
  L^{a,b,c}(q,w) = c(a^{-1} K_q(aw) - b f(q)) \ ,
\end{equation}
where $a,b,c > 0$ and $f$ is a scalar field on $\maxexpat \mu$.
The associated Euler-Lagrange equation reads
\begin{multline}\label{abcEL}
  a \acceleration
  q(t) + \velocity q(t) =  a^{-1}\left(\euler^{-a\velocity q(t) +
      K_{q(t)}(a \velocity q(t))} - \expectat {q(t)} {\euler^{-a\velocity q(t) +
      K_{q(t)}(a \velocity q(t))}}\right)
- \\
b \left(\euler^{-a\velocity q(t) +
      K_{q(t)}(a \velocity q(t))} \Grad f(q(t)) - \expectat {q(t)} {\euler^{-a\velocity q(t) +
        K_{q(t)}(a \velocity q(t))} \Grad f(q(t)) }\right) \ .
\end{multline}

For $\derivby t q(t) = q(t) v(t)$, the acceleration can be written as 
\begin{equation*}
\acceleration q(t) = \derivby t v(t) - \expectat {q(t)} {\derivby t
  v(t)} = \dot v(t) + \expectat {q(t)} {v(t)^2} \ .
\end{equation*}
The left-hand side of \cref{eq:EL-2} becomes
$a\left(\dot v(t) + \expectat {q(t)}{v(t)^2}\right) + v(t)$, while the
right-hand side becomes
\begin{multline*}
  a^{-1} \expectat {q(t)} {\euler^{av(t)}} \left(\euler^{-av(t)} -
    \expectat {q(t)} {\euler^{-av(t)}}\right) - \\ b \expectat {q(t)}
  {\euler^{av(t)}} \left(\euler^{-av(t)}\Grad f(q(t)) -
    \expectat {q(t)} {\euler^{-av(t)}\Grad f(q(t))}\right) \ .
\end{multline*}

Writing the expected values as sums, we obtain a system of $2N$
ordinary differential equations.  The system could be further reduced
to $2(N-1)$ equations between independent variables by using one of
the possible parametrizations of the Grassmanian manifold.

In \cref{simp2} in the article, we plot a solution of the Euler-Lagrangian equation for a KL Lagrangian parametrised by $a=1$ and $b=0$. The associated ODEs is given by
\begin{equation}\label{abc-ode}
  \left\{
    \begin{aligned}
    \derivby t q(x;t) &= q(x;t)v(x;t) \\
    \derivby t v(x;t) &= - v(x;t) - \frac{1}{N} \sum_y q(y;t)\,
    v(y;t)^2 - \\ &\phantom{=}  \frac{1}{N} \sum_y q(y;t)\, \euler^{v(y;t)}\left(\euler^{-v(x;t)} - \frac{1}{N} \sum_y q(y;t)\, \euler^{-v(y;t)}\right)
  \end{aligned}
  \right. \ .
\end{equation}

\subsection{\emph{Damped} Kullback-Leibler Lagrangian/Hamiltonian Systems under Ideal-Scaling Conditions - Figures 3,4}\label{jordano}

Here, we consider the case of the Euler-Lagrange equations for the damped KL Lagrangian under ideal scaling conditions (cf. \cref{bobo} in the article),
\begin{equation}\label{bobosm}
\acceleration q(t)  +(\euler^{\alpha_t}-\dot{\alpha}_t) \velocity q(t)= -
     \euler^{2\alpha_t+\beta_t} \left(\frac{q(t)}{\chi(t)}\Grad f(q(t)) -
       \expectat {q(t)} {\frac{q(t)}{\chi(t)}\Grad f(q(t))}\right) 
\end{equation}

Again, we have $\velocity q(t)=v(t)$ such that 
\begin{equation}
\derivby t q(x;t) = q(x;t) v(x;t) \ , \quad x \in \Omega \ .
\end{equation}
from which we get
\begin{multline}
\derivby t v(t) = \acceleration q(t) - \expectat {q(t)}{v(t)^2} = - (\euler^{\alpha_t}-\dot{\alpha}_t) v(t)  -  \frac{\euler^{2\alpha_t+\beta_t}\,\Grad f(q(t))}{\euler^{ \euler^{-\alpha_t}v(t)-K_q(\euler^{-\alpha_t}v(t))}} +\\  \expectat {q(t)}{\frac{\euler^{2\alpha_t+\beta_t}\,\Grad f(q(t))}{\euler^{ \euler^{-\alpha_t}v(t)-K_q(\euler^{-\alpha_t}v(t))}}} - \expectat {q(t)}{v(t)^2}
\end{multline}

Then, along with \cite{WibisonoE7351}, the choice of parameters, indexed by $p>0$ ($C>0$, a constant)
\begin{equation}
\label{parametr-sm}
\alpha_t= \log p-\log t  \qquad\quad 
\beta_t=  p \log t+\log C \qquad\quad 
\gamma_t=p \log t, 
\end{equation}
leads to the ODEs system
\begin{equation}\label{p-odee}
  \left\{
    \begin{aligned}
    \derivby t q(x;t) &= q(x;t)v(x;t) \\
    \derivby t v(x;t) &= - \frac{p+1}{t}\, v(t) -\frac{\,Cp^2t^{p-2}\, \Grad f(q(t))}{\euler^{\frac{t}{p}\, v(t)-K_q(\frac{t}{p}\,v(t))}} + \\ & \qquad \qquad \qquad  \frac{1}{N} \sum_x q(x;t)\Big(\frac{\,Cp^2t^{p-2}\, \Grad f(q(t))}{\euler^{\frac{t}{p}\, v(t)-K_q(\frac{t}{p}\,v(t))}}\Big) + \\ \nonumber
&  \qquad \qquad \qquad \frac{1}{N} \sum_x q(x;t)\, v(x;t)^2 \quad \quad x \in \Omega \ .
  \end{aligned}
  \right. 
\end{equation}

A solutions of the Lagrangian system above (for $p=2, C=0.5$ is shown the plot in \cref{simp3}.
\bigskip

Analogously, let us consider the the damped KL Hamiltonian system (cf. \cref{eq:Hamiltondamp} in the article), under ideal scaling condition and with the same parametrisation given in \cref{parametr-sm}.

Given $  \Derivby t {\eta}(t) = \frac{\dot q(t)}{q(t)} {\eta}(t) + \dot{{\eta}}(t)=\velocity q(t) \,{\eta}(t) +  \dot{{\eta}}(t)$,
we get a system of first order equations (see \cref{simp4})
\begin{equation}
  \label{eq:Hamiltondampeta-sm}
  \left \{
  \begin{aligned}
 \dot{{\eta}}(t) &= \euler^{\alpha_t}\, {\eta} - \euler^{\alpha_t+\gamma_t}\left(1+\euler^{-\gamma_t}\,\eta \right)\left(\log(1+\euler^{-\gamma_t}\,{\eta}) - \expectat {q(t)}
    {\log(1+\euler^{-\gamma_t}\,{\eta})}\right) \\ 
    & \quad - \euler^{\gamma_t+\alpha_t+ \beta_t} \Grad f(q) \\
     \dot q(t) & = \euler^{\alpha_t}\, q(t)\,\left(\log(1+\euler^{-\gamma_t}\,{\eta}) - \expectat {q(t)}
  {\log(1+\euler^{-\gamma_t}\,{\eta})}\right) \  .
  \end{aligned}
  \right.
\end{equation}

A solutions of the Hamiltonian ODEs system (for $p=2, C=0.5$) is plot in \cref{simp4} in the  article.

\bibliographystyle{amsplain}
\bibliography{tutto}

\end{document}